\documentclass[reqno]{amsart} 
\usepackage{epsfig}
\usepackage{graphicx}% Include figure files
\usepackage{bm}% bold math
\usepackage{amsmath}
\usepackage{amsthm}
\usepackage{amssymb}
\usepackage{amsfonts}
\usepackage{tikz}
\usepackage{epic}
\usepackage{eepic}
 \usepackage{times}
\usepackage[matrix,arrow]{xy}

\DeclareMathOperator*{\Res}{Res}
\newcommand{\BZ}{{\mathbb Z}}
\newcommand{\nc}{\newcommand}
\newcommand{\Z}{\mathbb Z}
\nc{\rnc}{\renewcommand} \nc{\beq}{\begin{equation}}
\nc{\eeq}{\end{equation}} \nc{\beqa}{\begin{eqnarray}}
\nc{\eeqa}{\end{eqnarray}}

\def \s{\underline{s}}

\def\stackreb#1#2{\ \mathrel{\mathop{#1}\limits_{#2}}}

\newcommand{\R}{\mathbb{R}}
\newcommand{\C}{\mathbb{C}}
\newtheorem{theorem}{Theorem}

\newtheorem{lemma}{Lemma}

\newtheorem{conjecture}{Conjecture}
\newtheorem{definition}{Definition}

\begin{document}

\begin{flushright} DESY 12-195 \end{flushright}

\title{A TQFT of Turaev--Viro type on shaped triangulations}

\author{Rinat Kashaev}
\address{University of Geneva\\
2-4 rue du Li\`evre, Case postale 64\\
 1211 Gen\`eve 4, Switzerland}
\email{rinat.kashaev@unige.ch}

\author{Feng Luo}
\address{Department of Mathematics\\
 Rutgers University\\
 Piscataway, NJ 08854, USA}
\email{fluo@math.rutgers.edu}

\author{Grigory Vartanov}
\address{DESY Theory\\
 Notkestrasse 85, 22603 Hamburg, Germany}
\email{grigory.vartanov@desy.de}

\thanks{Supported in part by Swiss National Science Foundation and United States National Science Foundation}

\begin{abstract}
A shaped triangulation is a finite triangulation of an oriented  pseudo three manifold where each tetrahedron carries dihedral angles of an ideal hyberbolic tetrahedron. To each shaped triangulation, we associate a quantum partition function in the form of an absolutely convergent state integral which is  invariant under shaped $3-2$ Pachner moves and  invariant with respect to  shape gauge transformations generated by total dihedral angles around internal edges through the Neumann--Zagier Poisson bracket. Similarly to Turaev--Viro theory, the state variables live on edges of the triangulation but take their values on the whole real axis. The tetrahedral weight functions  are composed of three  hyperbolic gamma functions in a way that they enjoy a manifest tetrahedral symmetry. We conjecture that for shaped triangulations of closed 3-manifolds, our partition function is twice the absolute value squared of the partition function of Techm\"uller TQFT defined by Andersen and Kashaev. This is similar to the known relationship between the Turaev--Viro and  the Witten--Reshetikhin--Turaev invarints of three manifolds. We also discuss interpretations of our construction in terms of three-dimensional supersymmetric field theories related to triangulated three-dimensional manifolds.
\end{abstract}

\maketitle
\section{Introduction}

Topological Quantum Field Theories were discovered and axiomatized by Atiyah~\cite{At}, Segal~\cite{S} and Witten~\cite{W}. First examples in $2+1$ dimensions were constructed by Reshetikhin and Turaev \cite{RT1,RT2,T} by using  the combinatorial framework of Kirby calculus, and by Turaev and Viro \cite{TV} by using the framework of triangulations and Pachner moves. The algebraic ingredients of both constructions come from the finite dimensional representation category of the quantum group $U_q(sl(2))$ at roots of unity. For example, the basic building elements in Turaev--Viro construction  are tetrahedral weight functions given by $6j$-symbols. These theories have been the subject of much subsequent investigation in the works of Blanchet, Habegger, Masbaum, Vogel, Barrett, Westbury, Turaev, Virelizier, Balsam, Kirillov and others \cite{BHMV1,BHMV2, BW, TVi, BK}. A related but somewhat different line of
development was initiated by Kashaev in \cite{K4} where a state sum invariant of
links in three manifolds was defined by using the combinatorics of charged triangulations  where the  charges are algebraic versions of dihedral angles of ideal hyperbolic tetrahedra in finite cyclic groups. This approach has been subsequently developed by Baseilhac, Benedetti, Geer, Kashaev, Turaev \cite{BB,GKT}. The common feature of all these theories is that the partition functions are always given by finite state sums.

On the other hand, the idea of partition functions of Turaev--Viro type originates from the work of  Ponzano and Regge \cite{PR} where, based on $SU(2)$ $6j$-symbols, a lattice version of quantum $2+1$ gravity was suggested, but  this theory was not complete and remained of restricted use  because of problems of convergence of infinite sums. Similar problems of convergence appear when one tries to construct combinatorial versions of quantum Chern--Simons theories with non-compact gauge groups. For example, a connected component of $PSL(2,\R)$ Chern--Simons theory is identified with Teichm\"uller space, and its  quantum theory corresponds to specific class of unitary  mapping class group representations in infinite dimensional Hilbert spaces  \cite{K1,CF}. Based on quantum Teichm\"uller theory, formal state-integral partition functions of triangulated three manifolds were defined by Hikami, Dimofte, Gukov, Lenells, Zagier, Dijkgraaf, Fuji, Manabe
\cite{H1,H2,DGLZ,DFM,D}, mostly for the purposes of quasi classical expansions, but the question of convergence remained largely open until a mathematically rigorous version of Teichm\"uller TQFT was suggested  in \cite{Andersen:2011bt}. The  convergence property of Teichm\"uller TQFT is due its specific underlying combinatorial setting: it is not just triangulations but shaped triangulations where  each tetrahedron carries dihedral angles of an ideal hyberbolic tetrahedron. Moreover,  the role of dihedral angles is two-fold: they not only provide absolute convergence of state integrals but they also implement the complete symmetry with respect to change of edge orientations. Although, shaped triangulations are similar to charged triangulations of \cite{K4}, the positivity condition of dihedral angles imposes important restrictions on construction of  topologically invariant partition functions.

The purpose of this paper is to suggest yet another TQFT based on combinatorics of shaped triangulations. As its basic building block is defined in terms of Faddeev's quantum dilogarithm \cite{F} and the absolute convergence of partition functions relies on the positivity of dihedral angles,  it is similar to the Teichm\"uller TQFT. As a consequence, we are still restricted in our abilities of constructing topologically invariant partition functions in the sense that the $2-3$ shaped Pachner move is not always applicable. On the other hand,  unlike the Teichm\"uller TQFT, our tetrahedral weight functions enjoy manifest tetrahedral symmetry and the partition function is well defined on any shaped triangulation without any extra topological restrictions.

Let us now describe  our construction in precise terms.

\subsection{States, state potentials, and state gauge invariance}
Let $Y$ be a CW-complex. Denote by $\Delta_i(Y)$ the set of $i$-dimensional cells of $Y$. A \emph{state} of $Y$ is a map
$s\colon \Delta_1(Y)\to \R$. A \emph{state potential} is a map $g\colon  \Delta_0(Y)\to \R$. Define a linear \emph{state gauge} map
\beq
b\colon \R^{\Delta_0(Y)}\to \R^{\Delta_1(Y)},\quad bg(e)=g(\partial_0e)+g(\partial_1e),
\eeq
where $\partial_ie$, $i\in\{0,1\}$, are the two end points of $e$ (they coincide if the edge is a loop).  A state is called \emph{pure gauge} if it finds itself in the image of the state gauge map.
 The pure gauge states constitute a vector subspace of the state space.

Let $S$ be a set. A function $f\colon \R^{\Delta_1(Y)}\to S$ is called \emph{state gauge invariant} at state $s$ if $f(s+bg)=f(s)$ for any state potential $g$.

A \emph{(state) gauge fixing} at vertex $v\in\Delta_0(Y)$ is a linear form  $\lambda$ on the vector space of states $\R^{\Delta_1(Y)}$ such that
\beq
\langle\lambda, bg\rangle=g(v),\quad \forall g\in\R^{\Delta_0(Y)}.
\eeq
Note that a gauge fixing at a vertex may not exist if the state gauge map is not injective.

In what follows, a real valued function defined on only a subset of vertices  will always be thought of as a state potential having zero values on the vertices where initially it was not defined.
\subsection{Shaped tetrahedra and their Boltzmann weights}
Let $T$ be an oriented tetrahedron embedded into $\R^3$ together with its
standard $CW$-complex structure. Let $\square(T)$ be
 the set of normal quadrilateral types (to be called \emph{quads})
  in $T$ which is in bijection with the set of pairs of opposite edges
  of $T$. We fix the action of $\Z/3\Z =\{1, \tau, \tau^2\}$ on
  $\square(T)$ so that the images of a quad $q$ under the action
  are $q, q'=\tau(q)$ and $q''=\tau^2(q)$
  % $(i,q)\mapsto q_i$,
   corresponding to the clockwise cyclic order of three edges around a vertex (as seen from the outside of the tetrahedron).
    We say $T$ is \emph{shaped tetrahedron} if it is provided with  a \emph{dihedral angle map} $\alpha\colon \square(T)\to ]0,\pi[$, such that
    $\alpha(q)+\alpha(q')+\alpha(q'')=\pi$. Associated to $\alpha$, the complex shape variables entering Thurston's hyperbolicity equations are given  by a map $z_\alpha\colon \square(T)\to \C\setminus\{0,1\}$  defined by the formula
\beq \label{shape-p}
z_\alpha(q)=e^{i\alpha(q)}\sin\alpha(q'')/\sin\alpha(q'). \eeq
 Any state $s\colon \Delta_1(T)\to \R$ induces a map $\tilde s\colon\square(T)\to\R$  defined by the formula $\tilde s(q)=s(e)+s(e')$, where the $e$ and $e'$ are the opposite edges separated by $q$.
To each pair $(T,s)$ consisting of a shaped tetrahedron $T$ and a
state $s$ of $T$, we associate the following \emph{Boltzmann
weight} \beq
B(T,s):=\prod_{q\in\square(T)}\gamma^{(2)}\left(\frac{\omega_1+\omega_2}\pi\alpha(q)+
\sqrt{-\omega_1\omega_2}(\tilde s(q')-\tilde
s(q''));\omega_1,\omega_2\right) \eeq where function
$\gamma^{(2)}(z;\omega_1,\omega_2)$ is defined below
in~\eqref{HGF} with $\omega_1, \omega_2 \in \mathbf{C}$ and
$\omega_1/\omega_2 \notin (-\infty, 0]$.   It is easily verified
that this Boltzmann weight is state gauge invariant at any state.

\subsection{Shaped triangulations and their Boltzmann weights}
A \emph{triangulation} is an oriented pseudo 3-manifold obtained from finitely many tetrahedra in $\R^3$ by gluing them along triangular faces through orientation reversing affine $CW$-homeomorphisms. Any triangulation $X$ is naturally a $CW$-complex and its boundary $\partial X$ is the $CW$-subcomplex composed of unglued triangular faces.  We will use the following notation:
\beq
\Delta_i(\mathring{X}):=\Delta_i(X)\setminus\Delta_i(\partial X).
\eeq

A \emph{shaped triangulation} is a triangulation where all tetrahedra are shaped.
Similarly to the case of one shaped tetrahedron, to each pair $(X,s)$ consisting of a shaped triangulation $X$ and a state $s$ of $X$, we associate a Boltzmann weight
\beq\label{eq:bw}
B(X,s):=\prod_{T\in \Delta_3(X)} B\left(T,s\vert_{\Delta_1(T)}\right).
\eeq
Again, this Boltzmann weight is state gauge invariant at any state.

\subsection{The partition function of shaped triangulations}
A \emph{boundary state} of a trianguation $X$ is a state of its boundary. We have the natural linear restriction  map from the vector space of states of $X$ to the vector space of its  boundary states
\[
\partial\colon \R^{\Delta_1(X)}\to\R^{\Delta_1(\partial X)},
\]
and for any boundary state $s$,  we have a canonical identification of the preimage $\partial^{-1}(s)$ with the linear space $\R^{\Delta_1(\mathring{X})}$  of real valued functions on the interior edges of $X$.

A  \emph{state gauge fixing} in the  interior of a triangulation $X$ is a collection
\beq
\lambda=\{\lambda_v\}_{v\in\Delta_0(\mathring{X})}
\eeq of gauge fixings at all interior vertices. Notice that for any triangulation the state gauge map is injective and state gauge fixings exist at any vertex.

To any triple $(X,s,\lambda)$, where $X$ is a  shaped triangulation, $s$ is a boundary state  of $X$, and $\lambda$ is a state gauge fixing in the interior of $X$, we associate a \emph{partition  function}
\begin{equation}\label{eq:pf}
W_b(X,s,\lambda):=\int_{\partial^{-1}(s)} B(X,t)\delta\left(\langle\lambda,t\rangle\right) dt,
\end{equation}
where
\beq
\delta(\langle\lambda,t\rangle):=\prod_{v\in \Delta_0(\mathring{X})}
\delta\left(\langle\lambda_v,t\rangle\right),\quad dt:=\prod_{v\in \Delta_1(\mathring{X})}dt(e)
\eeq
and  $b =\sqrt{\frac{\omega_1}{\omega_2}}$.
The main result of this paper is the following theorem where we use the notions of \emph{shaped $3-2$  Pachner moves} and  \emph{shape gauge transformations} considered in \cite{Andersen:2011bt}.
\begin{theorem}\label{main}
The partition function $W_b(X,s,\lambda)$ is an absolutely convergent integral independent of the choice of the state  gauge fixing $\lambda$, invariant under shaped  $3-2$ Pachner moves, and invariant under  the shape gauge transformations induced by interior edges.
\end{theorem}
Several examples of explicit calculations make us to believe that for shaped triangulations of closed 3-manifolds, when the Teichm\"uller TQFT is defined as well, our partition function is twice  the absolute value squared of the partition function of the Techm\"uller TQFT. This is similar to the known relationship between the Turaev--Viro and  the Witten--Reshetikhin--Turaev invarints of three manifolds.
\begin{conjecture}\label{conj}
Let $(X,\ell_X)$ be an admissible shaped levelled branched triangulation of a closed oriented compact three manifold in the sense of \cite{Andersen:2011bt}. Then the following equality holds true
\beq
2\left|F_\hbar(X,\ell_X)\right|^2=W_b(X,\lambda)
\eeq
where $\hbar=(b+b^{-1})^{-2}\in\R_{>0}$.
\end{conjecture}

The rest of this paper is organized as follows. Section~\ref{proof} contains the proof of the main Theorem~\ref{main}. In Section~\ref{s:pent}, we derive the pentagon identity which underlies the invariance of our partition function with respect to shaped $3-2$ Pachner move from the elliptic beta-integral.  In Section~\ref{examples} we provide examples of concrete calculations which justify Conjecture~\ref{conj}. Section~\ref{3dSUSY} is devoted to some considerations from the perspective of  $3d$ supersymmetric field theories. Namely, based on our construction we get a class of $3d$ supersymmetric field theories defined on a squashed three-sphere $S_b^3$ related to triangulated three-dimensional manifolds. The latter relation is known as $3d/3d$ correspondence which is the topic of recent study \cite{Terashima:2011qi,Dimofte:2011ju,Dimofte:2011py,Teschner:2012em}. Appendices contain some technical information on the special functions used.

\subsection*{Acknowledgements} We would  like to thank the organizers of following  events during which our collaboration
in this work  took place: the conference ``Groupes de diff\'eotopie et topologie quantique", Strasbourg, 25--29 June 2012;
 the ``International congress on mathematical physics", Aalborg, 6--11 August 2012; the
 workshop ``New Perspectives in Topological Field Theories",  Hamburg,
 27--31 August, 2012. We also appreciate the supports by Swiss National Science Foundation and
ITGP (Interactions of Low-Dimentional Topology and Geometry with
Mathematical Physics), an ESF RNP, and US National Science
Foundation.

\section{Proof of Theorem~\ref{main}}\label{proof}

\begin{lemma}\label{l:1}
Let $X$ be a shaped triangulation, and let $s$ and $s'$ be states of $X$ such that $B(X,s'+ts)=B(X,s')$ for any $t\in\R$. Then the state $s$ is in the image of the state  gauge map.
\end{lemma}
\begin{proof}
 By a straightforward verification,  the statement of the lemma is true if $X$ is a disjoint union of unglued tetrahedra. Thus, it suffices to prove that if triangulation $X$ is obtained   from  a triangulation $Y$  by  identification of two  triangular faces $f$ and $f'$, and the statement of the lemma is true for $Y$,  then it is also true for $X$.

Denote by $p\colon Y\to X$ the identification projection, and by $p^*\colon \R^{\Delta_i(X)}\to  \R^{\Delta_i(Y)}$ the corresponding pull-back maps. Let  $s$ and $s'$ be states of $X$ such that
\beq\label{eq:ginv}
B(X,s'+ts)=B(X,s'),\quad \forall t\in\R.
\eeq
 Using the fact that $B(X,r)=B(Y,p^*(r)) $ for any state $r$ of $X$, equation~\eqref{eq:ginv} is equivalent to
\beq
B(Y,p^*(s')+tp^*(s))=B(Y,p^*(s')),\quad \forall t\in\R.
\eeq
As we assume that the statement of the lemma is true for $Y$, there exists  $g\in\R^{\Delta_0(Y)}$ such that $p^*(s)=bg$. Let us show that there exists $g'\in\R^{\Delta_0(X)}$ such that $g=p^*(g')$. Indeed, let triangles $f$ and $f'$ have respective vertices $v_i$ and $v'_i$ and edges $e_i$ and $e'_i$ for $i\in\{1,2,3\}$ such that
\beq
\partial e_i=\{v_j,v_k\},\quad \partial e'_i=\{v'_j,v'_k\},\quad \{i,j,k\}=\{1,2,3\},
\eeq
and
\beq
p(e_i)=p(e'_i),\quad p(v_i)=p(v'_i),\quad i\in\{1,2,3\}.
\eeq
That means that when applied to edges $e_i$ and $e'_i$, the equality $p^*(s)=bg$ gives
\beq
g(v_j)+g(v_k)=g(v'_j)+g(v'_k)\Leftrightarrow g(v_i)-g(v_i')=\xi:=\sum_{m=1}^3(g(v_m)-g(v'_m)).
\eeq
Taking sum over $i$ in the last equation, we obtain
\(
\xi=3\xi\Leftrightarrow\xi=0
\)
which implies that $g(v_i)=g(v'_i)$ for any $i\in\{1,2,3\}$, i.e. $g=p^*(g')$.
Thus, we have the equality $p^*(s)=bp^*(g')=p^*(bg')$, and  as  $p^*$ is injective, we conclude that $s=bg'$.
\end{proof}
\begin{proof}[Proof of Theorem~\ref{main}]
By injectivity of  the state gauge map in the case of triangulations and Lemma~\ref{l:1}, the state gauge map image of the group $\R^{\Delta_0(\mathring{X})}$  is the maximal translation subgroup of the state space of $X$ which leaves invariant the boundary state $s$ and the Boltzmann weight $B(X,t)$. On the other hand,  the product of delta functions $\delta(\langle\lambda,t\rangle)$ restricts the integral to a hyperplane in the space $\R^{\Delta_1(\mathring{X})}\simeq b^{-1}(s)$ which intersects any orbit of this group action in a unique point, while the Boltzmann weight exponentially decays along  any direction  in this  hyperplane. This implies that the integral in~\eqref{eq:pf} is absolutely convergent.

Independence on the choice of the state gauge fixing $\lambda$ easily follows through the use of a simplest finite-dimensional version of the Faddeev--Popov trick in path integrals for gauge invariant systems\footnote{Similarly to QED, our system is linear and the Faddeev--Popov determinant is trivial so that no ghosts are needed.} \cite{FP}. Indeed, if $t$ is a state of $X$ and $\lambda'$ a state gauge fixing  in the interior of $X$, then we have the identity
\beq\label{eq:fp}
1=\int_{\R^{\Delta_0(\mathring{X})}}\delta\left(\langle\lambda',t+bg\rangle\right)dg,
\eeq
where
\beq
dg:=\prod_{v\in \Delta_0(\mathring{X})}dg(v).
\eeq
Inserting \eqref{eq:fp} into \eqref{eq:pf}, exchanging the order of integrations, shifting the  integration state variables,  using the gauge invariance of the Boltzmann weight, again exchanging the order of integrations, and again using identity~\eqref{eq:fp} with $\lambda'$ replaced by $\lambda$,  we  obtain
\begin{multline}
W_b(X,s,\lambda)=
\int_{\partial^{-1}(s)} B(X,t)\delta\left(\langle\lambda,t\rangle\right)dt=\\
\int_{\partial^{-1}(s)} B(X,t)\delta\left(\langle\lambda,t\rangle\right)\left( \int_{\R^{\Delta_0(\mathring{X})}}\delta\left(\langle\lambda',t+bg\rangle\right)dg\right)dt\\
=
\int_{\R^{\Delta_0(\mathring{X})}}\left(\int_{\partial^{-1}(s)} B(X,t)\delta\left(\langle\lambda,t\rangle\right) \delta\left(\langle\lambda',t+bg\rangle\right)dt\right)dg\\
=
\int_{\R^{\Delta_0(\mathring{X})}}\left(\int_{\partial^{-1}(s)} B(X,t-bg)\delta\left(\langle\lambda,t-bg\rangle\right) \delta\left(\langle\lambda',t\rangle\right)dt\right)dg\\
=
\int_{\R^{\Delta_0(\mathring{X})}}\left(\int_{\partial^{-1}(s)} B(X,t)\delta\left(\langle\lambda,t-bg\rangle\right) \delta\left(\langle\lambda',t\rangle\right)dt\right)dg\\
=
\int_{\partial^{-1}(s)} B(X,t)\delta\left(\langle\lambda',t\rangle\right)\left( \int_{\R^{\Delta_0(\mathring{X})}}\delta\left(\langle\lambda,t-bg\rangle\right)dg\right)dt\\
=
\int_{\partial^{-1}(s)} B(X,t)\delta\left(\langle\lambda',t\rangle\right)dt=W_b(X,s,\lambda').
\end{multline}

 Invariance under $3-2$ shaped Pachner moves is a consequences of the shaped pentagon identity for the tetrahedral Boltzmann weights, which in its turn is equivalent to identity~\eqref{Pent1}, provided the relevant integration variable does not enter the product of delta-functions
$\delta\left(\langle\lambda,t\rangle\right)$. This condition can always be satisfied by appropriate choice of $\lambda$.

Finally, the gauge transformation in the space of dihedral angles induced by an edge $e$, see \cite{Andersen:2011bt}, is equivalent to an imaginary shift of the integration variable $s(e)$, which, by using  the holomorphicity of the Boltzmann weights, can be compensated by an imaginary shift of the integration path in the complex $s(e)$-plane.
\end{proof}

\section{Pentagon identities from elliptic beta-integral}\label{s:pent}

Let us start from Spiridonov's elliptic beta-integral \cite{S1}
\beq \label{BI}
\kappa \int_{\mathbb{T}} \frac{\prod_{i=1}^6 \Gamma(s_i z^{\pm1};p,q)}{\Gamma(z^{\pm2};p,q)} \frac{dz}{2 \pi \textup{i} z} = \prod_{1 \leq i < j \leq 6} \Gamma(s_i s_j;p,q),
\eeq
where parameters $\s = \{s_1,\ldots,s_6\}$ satisfy the so-called balancing condition $\prod_{i=1}^6 s_i = pq$. Here
$$
\kappa \ = \ \frac{(p;p)_\infty (q;q)_\infty}{2},
$$
where $(z;p)_\infty = \prod_{i=0}^\infty (1-zp^i)$. Also in (\ref{BI}) the building block is the elliptic gamma function defined as
\beq \label{EGF}
\Gamma(z;p,q) = \prod_{i,j=0}^\infty \frac{1-z^{-1} p^{i+1} q^{j+1}}{1-zp^iq^j},
\eeq
with $|z|<1$ and two basis parameters $|p|,|q|<1$. Here we use the following useful conventions
$$
\Gamma(a,b;p,q) = \Gamma(a;p,q) \Gamma(b;p,q), \ \ \ \Gamma(az^{\pm1};p,q) = \Gamma(az;p,q) \Gamma(az^{-1};p,q).
$$

The elliptic gamma function has the following limit when all its parameters and the two basis paramteres simultaneously go to unity \cite{Rui}
\beq \label{lim} \Gamma(e^{2 \pi \textup{i} r z};e^{2 \pi \textup{i} r \omega_1},
e^{2 \pi \textup{i} r \omega_2}) \stackreb{=}{r \rightarrow 0}
e^{-\pi \textup{i}(2z-\omega_1-\omega_2)/12r} \gamma^{(2)}(z;\omega_1,\omega_2),\eeq
where on the right hand side one has the so-called hyperbolic gamma function
\beq \label{HGF}
\gamma^{(2)}(u;\omega_1,\omega_2) = e^{-\pi \textup{i}
B_{2,2}(u;\omega_1,\omega_2)/2} \frac{(e^{2 \pi i u/\omega_1}
\widetilde{q};\widetilde{q})_\infty}{(e^{2 \pi i
u/\omega_2};q)_\infty},\eeq
with the redefined basis parameters
\beqa\nonumber
 q = e^{2 \pi i \omega_1/\omega_2},\qquad  \widetilde{q} = e^{-2 \pi i
\omega_2/\omega_1},\eeqa and $B_{2,2}(u;\omega_1,\omega_2)$
denoting the
second order Bernoulli polynomial, \beq B_{2,2}(u;\omega_1,\omega_2) =
\frac{u^2}{\omega_1\omega_2} - \frac{u}{\omega_1} -
\frac{u}{\omega_2} + \frac{\omega_1}{6\omega_2} +
\frac{\omega_2}{6\omega_1} + \frac 12.\eeq
The conventions,
\[
\gamma^{(2)}(a,b;\omega_1,\omega_2) \equiv
\gamma^{(2)}(a;\omega_1,\omega_2) \gamma^{(2)}(b;\omega_1,\omega_2),
\]
and
\[
\gamma^{(2)}(a\pm u;\omega_1,\omega_2) \equiv
\gamma^{(2)}(a+u;\omega_1,\omega_2)
\gamma^{(2)}(a-u;\omega_1,\omega_2),
\]
are applied further.

Also we are going to use the following reduction of the hyperbolic gamma function \cite{Rui}
\beq \label{GF}
\gamma^{(2)}(z;\omega_1,\omega_2) \stackreb{=}{\omega_2 \rightarrow \infty}
\left(\frac{\omega_2}{2 \pi \omega_1}\right)^{\frac 12 - \frac{z}{\omega_1}}
\frac{\Gamma(z/\omega_1)}{\sqrt{2 \pi}},
\eeq
where $\Gamma(u)$ is the usual gamma function. The relation between the hyperbolic gamma function and Faddeev's quantum dilogarithm is presented  in the Appendix where we collect all the definitions and properties of the special functions.

\subsection{The first pentagon}
Let us start from elliptic beta-integral (\ref{BI}) and reparametrize parameters as
$$
s_i = e^{2 \pi \textup{i} v \alpha_i}, i=1,\ldots,6; \ z= e^{2 \pi \textup{i} v u}; \ p = e^{2 \pi \textup{i} v \omega_1}; \ q = e^{2 \pi \textup{i} v \omega_2},
$$
and use the limit of elliptic gamma function (\ref{lim}) to get
\beq \label{HBI}
\frac 12 \int_{-\textup{i} \infty}^{\textup{i} \infty} \frac{\prod_{i=1}^6 \gamma^{(2)}(\alpha_i \pm u;\omega_1,\omega_2)}{\gamma^{(2)}(\pm 2u;\omega_1,\omega_2)} \frac{du}{\textup{i} \sqrt{\omega_1\omega_2}} = \prod_{1 \leq i < j \leq 6} \gamma^{(2)}(\alpha_i + \alpha_j;\omega_1,\omega_2),
\eeq
where the balancing condition becomes $\sum_{i=1}^6 \alpha_i = \omega_1+\omega_2$.

To get the new form of the pentagon one should proceed as follows. Let us take reparameterization \cite{Spiridonov:2010em}
\beq
\alpha_i = \mu + a_i, \alpha_{i+3} = -\mu + b_i, i=1,2,3,
\eeq
which preserves the balancing condition and consider the limit $\mu \rightarrow \infty$.
We use the inversion relation
\beq \label{Inv}
\gamma^{(2)}(z,\omega_1+\omega_2-z;\mathbf{\omega_1},
\mathbf{\omega_2}) = 1 \eeq
 and the asymptotic formulas \beqa \makebox[-1.25em]{} \lim_{u
\rightarrow \infty} e^{\frac{\pi \textup{i}}{2}
B_{2,2}(u;\mathbf{\omega})} \gamma^{(2)}(u;\mathbf{\omega})
& = & 1, \text{ \ \ for } \text{arg }\omega_1 < \text{arg } u < \text{arg }\omega_2 + \pi, \nonumber \\
\lim_{u \rightarrow \infty}e^{-\frac{\pi \textup{i}}{2} B_{2,2}(u;\mathbf{\omega})} \gamma^{(2)}(u;\mathbf{\omega})
& = & 1, \text{  \ \ for } \text{arg } \omega_1 - \pi < \text{arg } u <
\text{arg }\omega_2,
\eeqa
and shifting the integration variable $u \rightarrow u+\mu$ to get
\beq \label{Pent1_p}
\int_{-\textup{i} \infty}^{\textup{i} \infty} \prod_{i=1}^3 \gamma^{(2)}(a_i - u, b_i + u;\omega_1,\omega_2) \frac{du}{\textup{i} \sqrt{\omega_1\omega_2}} = \prod_{i,j=1}^3 \gamma^{(2)}(a_i+b_j;\omega_1,\omega_2),
\eeq
with $\sum_{i=1}^3 (a_i+b_i) = \omega_1+\omega_2$.
Let us introduce now the following function \beq \label{HB}
\mathcal{B}(x,y) =
\frac{\gamma^{(2)}(x,y;\omega_1,\omega_2)}{\gamma^{(2)}(x+y;\omega_1,\omega_2)},
\eeq after which we rewrite (\ref{Pent1_p}) as \beq \label{Pent1}
\int_{-\textup{i} \infty}^{\textup{i} \infty} \prod_{i=1}^3
\mathcal{B}(a_i - u,b_i + u) \frac{du}{\textup{i}
\sqrt{\omega_1\omega_2}} = \mathcal{B}(a_2+b_1,a_3+b_2)
\mathcal{B}(a_1+b_2,a_3+b_1), \eeq where we used the inversion
relation for the hyperbolic gamma function. The geometric meaning of
the pentagon relation (\ref{Pent1}) can be seen in the figure
below.  Note that by the inversion formula (\ref{Inv}), we have
$$\mathcal{B}(x,y) =\gamma^{(2)}(x;\omega_1,\omega_2)
\gamma^{(2)}(y; \omega_1,
\omega_2)\gamma^{(2)}(\omega_1+\omega_2-x-y;\omega_1, \omega_2).$$
Therefore, the right hand-side of (\ref{Pent1}) is the
Boltzmann weight for the union of two tetrahedra and the
left-hand-side of (\ref{Pent1}) is the integration of the
Boltzmann weight of three tetrahedra.
\begin{figure}[ht!]
\centering
\includegraphics[scale=0.55]{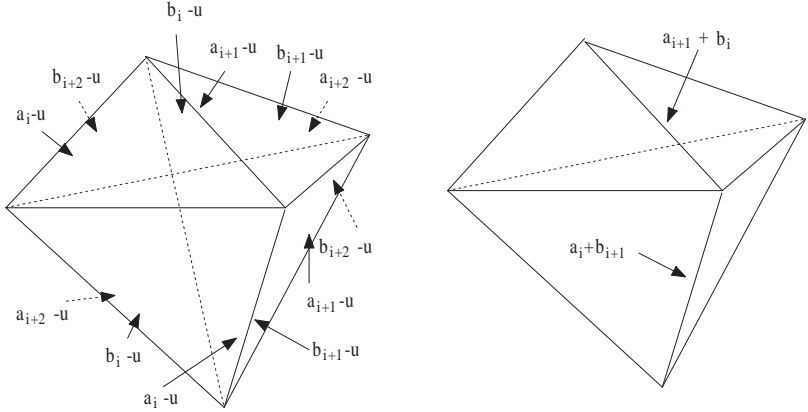}
\caption{2-3 moves} \label{figure 1}
\end{figure}
%\begin{figure}[ht!]\vspace{0.3cm}
%\begin{center}
%\leavevmode \epsfxsize=4cm \epsffile{1.eps}
%\end{center}\vspace{0.2cm}
%\caption{$2-3$ Pachner move.}\label{figure 1}
%\end{figure}

From (\ref{Pent1}) we see that the function $\mathcal{B}$ satisfies pentagon identity. It will be natural to suggest that the original elliptic beta-integral should also satisfy some kind of pentagon identity, but to our knowledge it is not realized so far. Recently in papers \cite{Bazhanov:2011mz} it was realized that the elliptic beta-integral satisfies the Yang-Baxter star-triangle relation (see also \cite{Spiridonov:2010em}) with the Boltzman weight $W_\alpha(x,y)=\Gamma(e^\alpha x^{\pm1} y^{\pm1};p,q)$. So, instead of $2-3$ Pachner move the elliptic beta-integral (\ref{BI}) satisfies $3-3$ Pachner move \cite{Korepanov} which might be relevant for construction of quantum invariants  of four-dimensional manifolds. Moreover,  the elliptic hypergeometric integrals describe specific partition functions of $4d$ $\mathcal{N}=1$ SYM theories known as superconformal indices \cite{Dolan:2008qi,SV2}. Combining these facts together one expects that triangulations of four-dimensional manifolds can be connected to four-dimensional $\mathcal{N}=1$ supersymmetric field theories.

One has the following orthogonality relations for the $\mathcal{B}$ function:
\beq \label{ORH}
\int_{\R} \mathcal{B}(a - \textup{i} u,b + \textup{i} u) \mathcal{B}(-a - \textup{i} u,-b + \textup{i} u) \frac{du}{\sqrt{\omega_1\omega_2}} = 2 \textup{i} \sqrt{\omega_1 \omega_2} \delta(a-b),
\eeq
\beq
\nonumber \makebox[-2em]{}\int_{\R} \mathcal{B}(a - \textup{i} u,a + \textup{i} u) \mathcal{B}(-a - \textup{i} (u+b),-a + \textup{i} (u+b)) \frac{du}{\sqrt{\omega_1\omega_2}} = 2 \sqrt{\omega_1 \omega_2} \delta(b).
\eeq

\subsection{The second pentagon}
Let us rewrite (\ref{Pent1}) as
\begin{multline} \label{Pent11}
\int_{-\textup{i} \infty}^{\textup{i} \infty} \frac{\prod_{i=1}^3 \gamma^{(2)}(a_i - u;\omega_1,\omega_2)  \prod_{i=1}^2 \gamma^{(2)}(b_i + u;\omega_1,\omega_2)}{\gamma^{(2)}(\sum_{i=1}^3 a_i + b_1 + b_2 - u;\omega_1,\omega_2)} \frac{du}{\textup{i} \sqrt{\omega_1\omega_2}}\\ = \prod_{i,j=1}^3 \gamma^{(2)}(a_i+b_j;\omega_1,\omega_2),
\end{multline}
Applying the limit $\omega_2 \rightarrow \infty$ to (\ref{Pent11}) and using (\ref{GF}) we get
\beqa \label{Pent2}
&& \int_{-\textup{i} \infty}^{\textup{i} \infty} B(a_1+u,b_1-u) B(a_2+u,b_2-u) B(a_3+u,a_1+a_2+b_1+b_2) \frac{du}{2 \pi \textup{i}} \nonumber \\ && \makebox[6em]{} = B(a_2+b_1,a_3+b_2) B(a_1+b_2,a_3+b_1),
\eeqa
where $B(x,y)$ is the usual beta-function
\beq \label{B}
B(x,y) \ = \ \frac{\Gamma(x) \Gamma(y)}{\Gamma(x+y)}.
\eeq

Taking the limit $\omega_2 \rightarrow \infty$ in (\ref{ORH}) one gets analogous orthogonality relations for the beta-integral. %for the usual beta-function

\section{Examples of calculations}\label{examples}
We will use the following notation
\beq
\Delta:=(\omega_1+\omega_2)/\pi,\quad \nabla:=\sqrt{\omega_1\omega_2},
\eeq
and also
\beq
u(x):=c_b\left(1-\frac x\pi\right).
\eeq
Further we will use the $\psi$ function defined as
\beq\label{eq:SBP}
\psi(x,y):=\Psi(x,-x,y)=\int_\R\frac{\Phi_b(t+x)}{\Phi_b(t-x)}e^{2\pi i yt}dt
\eeq
see also \eqref{eq:Psi}. We also have the equality
\beq
\mathcal{B}\left(\Delta \alpha+i\nabla x,\Delta\beta+i\nabla y\right) = \psi\left(u(\alpha)+\frac x2,y-2c_b\frac\beta\pi\right)
\eeq
which is equivalent to \eqref{eq:BPsi}.

\subsection{One vertex H-triangulation of $(S^3,3_1)$}
Following \cite{Andersen:2011bt}, we have one  tetrahedron $T$ with linearly ordered vertices (enumerated by $0,1,2,3$) and with the face identifications
\beq \label{ident}
\partial_i T \simeq \partial_{3-i} T, \ i \in {0,1},
\eeq
represented by diagram
\begin{figure}[ht!]
\centering
\includegraphics[scale=1]{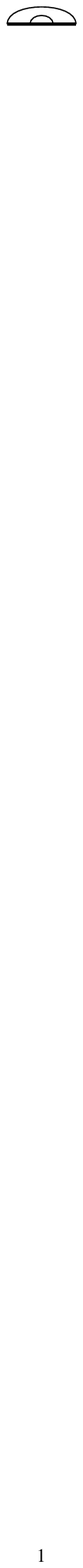}
%\caption{2-3 moves} \label{figure 1}
\end{figure}

The quotient space $X$ is a triangulation of $S^3$ with only one vertex $v$ and two edges: $e_1$ knotted like trefoil and having as preimage the only egde $03$ of $T$, and  $e_2$ having  as preimages all other five edges of $T$.
The Boltzmann weight reads
\beq
B(X,s)=\mathcal{B}\left(\Delta \alpha_0,\Delta \alpha_1 + i\nabla (s_2-s_1)\right)
\eeq
where $\alpha_i:=\alpha(q_i)$ with $q\in\square(T)$ being  the quad corresponding to the opposite edge pair $(03,12)$ of $T$, and $s_i:=s(e_i)$.
Choosing the gauge fixing map $\lambda$ so that $\langle \lambda_v,s\rangle=s_1/2$, we obtain the following  integral  for the partition function:
\begin{multline} \label{int1}
W_b(S^3,3_1) = \int_{\R^2} \mathcal{B}\left(\Delta \alpha_0,\Delta \alpha_1 + i\nabla (s_2-s_1)\right) \delta(s_1/2) ds_1 ds_2\\=
2 \int_{\R} \mathcal{B}\left(\Delta \alpha_0,\Delta \alpha_1 + i\nabla s_2\right) ds_2
\end{multline}
which under substitution of \eqref{eq:SBP} is calculated as follows
\begin{multline}
W_b(S^3,3_1) =
2 \int_{\R} \psi\left(u(\alpha_0),s_2-2c_b\frac{\alpha_1}\pi\right) ds_2\\=
2 \int_{\R^2}\frac{\Phi_b( t+u(\alpha_0))}{\Phi_b( t-u(\alpha_0))}e^{2\pi i\left(s_2-2c_b\frac{\alpha_1}\pi\right)t} dt ds_2=
2 \int_{\R}\frac{\Phi_b( t+u(\alpha_0))}{\Phi_b( t-u(\alpha_0))}\delta(t)e^{-4ic_b\alpha_1t} dt\\
=2\frac{\Phi_b( u(\alpha_0))}{\Phi_b(-u(\alpha_0))}=2\left|\Phi_b( u(\alpha_0))\right|^2.
\end{multline}
As in \cite{Andersen:2011bt}, the partition function diverges in the H-balanced limit $\alpha_0\to0$, so that it makes sense only to consider the ratios of partition functions. Thus, we define the renormalized partition function $\tilde W_b(S^3,3_1)=1$.

\subsection{One vertex H-triangulation of $(S^3,4_1)$}

In the graphical notation of \cite{Andersen:2011bt}, let $X$ be given by the diagram
\begin{figure}[ht!]
\centering
\includegraphics[scale=1]{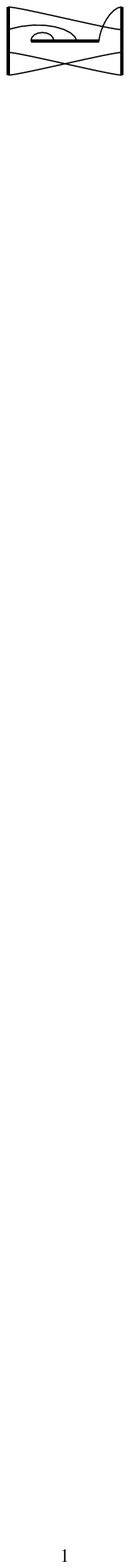}
%\caption{2-3 moves} \label{figure 1}
\end{figure}
where the figure-eight knot is represented by the edge of the central tetrahedron $T$ connecting the maximal and the next to maximal vertices. This  H-triangulation of $(S^3,4_1)$ consists of two positive $T$ and $T_L$ to the left from the central tetrahedron and one negative $T_R$-- to the right. One has the following identification of the faces
\beqa \label{ident1}
&& \partial_0 T \simeq \partial_1 T, \ \partial_2 T \simeq \partial_1 T_L, \ \partial_3 T \simeq \partial_3 T_R, \nonumber \\ && \partial_0 T_L \simeq \partial_2 T_R, \ \partial_2 T_L \simeq \partial_0 T_R, \ \partial_3 T_L \simeq \partial_1 T_R.
\eeqa
Identifying the corresponding edges, one gets
\beqa \label{ident_edge}
&& z \equiv x_{13} = x_{03} = x_{23}^{L} = x_{03}^L = x_{13}^{R}, \nonumber \\ && y \equiv x_{12} = x_{02} = x_{01}^L = x_{02}^R = x_{12}^{R}, \nonumber \\ && x \equiv x_{01} = x_{02}^L = x_{12}^{L} = x_{13}^{L} = x_{01}^R = x_{03}^R = x_{23}^{R},
\eeqa
and one has also the edge $x' \equiv x_{23}$. Then for the partition function we have
\begin{multline*}
W_b(S^3,4_1)=\int_{\R^4}\mathcal{B}(\Delta\alpha_1,\Delta\alpha_2+\textup{i}\nabla(y+z-x-x'))\\
\times\mathcal{B}(\Delta\beta_1+\textup{i}\nabla(x-z),\Delta\beta_2+\textup{i}\nabla(x-y))\\
\times\mathcal{B}(\Delta\gamma_1+\textup{i}\nabla(x-z),\Delta\gamma_2+\textup{i}\nabla(x-y))\delta(x/2)dxdydzdx'\\
=2\int_{\R^3}\mathcal{B}(\Delta\alpha_1,\Delta\alpha_2+\textup{i}\nabla(y+z-x'))\\
\times\mathcal{B}(\Delta\beta_1-\textup{i}\nabla z,\Delta\beta_2-\textup{i}\nabla y)\mathcal{B}(\Delta\gamma_1-\textup{i}\nabla z,\Delta\gamma_2-\textup{i}\nabla y)dydzdx'\\
=2\tilde W_b(S^3,4_1)\int_{\R}\mathcal{B}(\Delta\alpha_1,\Delta\alpha_2+\textup{i}\nabla t)dt
=2\tilde W_b(S^3,4_1)|\Phi_b(u(\alpha_1))|^2
\end{multline*}
where
\begin{multline}
\tilde W_b(S^3,4_1):=\frac{ W_b(S^3,4_1)}{2|\Phi_b(u(\alpha_1)|^2}\\
=\int_{\R^2}\mathcal{B}(\Delta\beta_1-\textup{i}\nabla z,\Delta\beta_2-\textup{i}\nabla y)\mathcal{B}(\Delta\gamma_1-\textup{i}\nabla z,\Delta\gamma_2-\textup{i}\nabla y)dydz.
\end{multline}
 Now, using~\eqref{eq:SBP} and \eqref{eq:Psi} we continue as follows:
\begin{multline*}
\tilde W_b(S^3,4_1)=\int_{\R^2}\psi\left(u(\beta_1)-\frac z2,-2c_b\frac{\beta_2}\pi-y\right)\psi\left(u(\gamma_1)-\frac z2,-2c_b\frac{\gamma_2}\pi-y\right)dydz\\=
\int_{\R^4} \frac{\Phi_b\left(u(\beta_1)-\frac z2+s\right)\Phi_b\left(u(\gamma_1)-\frac z2+t\right)}{\Phi_b\left(-u(\beta_1)+\frac z2+s\right)\Phi_b\left(-u(\gamma_1)+\frac z2+t\right)}e^{-2\pi \textup{i}(s+t)y-4\textup{i}c_b(s\beta_2+t\gamma_2)}dsdtdydz\\=
\int_{\R^3} \frac{\Phi_b\left(u(\beta_1)-\frac z2+s\right)\Phi_b\left(u(\gamma_1)-\frac z2+t\right)}{\Phi_b\left(-u(\beta_1)+\frac z2+s\right)\Phi_b\left(-u(\gamma_1)+\frac z2+t\right)}e^{-4 \textup{i}c_b(s\beta_2+t\gamma_2)}\delta(s+t)dsdtdz\\=
\int_{\R^2} \frac{\Phi_b\left(u(\beta_1)-\frac z2-t\right)\Phi_b\left(u(\gamma_1)-\frac z2+t\right)}{\Phi_b\left(-u(\beta_1)+\frac z2-t\right)\Phi_b\left(-u(\gamma_1)+\frac z2+t\right)}e^{4\textup{i}c_bt(\beta_2-\gamma_2)}dtdz\\=\{t\mapsto t+\frac z2\}=
\int_{\R^2} \frac{\Phi_b\left(u(\beta_1)-z-t\right)\Phi_b\left(u(\gamma_1)+t\right)}{\Phi_b\left(-u(\beta_1)-t\right)\Phi_b\left(-u(\gamma_1)+z+t\right)}e^{4\textup{i}c_b(t+\frac z2)(\beta_2-\gamma_2)}dtdz\\=\{z\mapsto z-t\}=
\int_{\R^2} \frac{\Phi_b\left(u(\beta_1)-z\right)\Phi_b\left(u(\gamma_1)+t\right)}{\Phi_b\left(-u(\beta_1)-t\right)\Phi_b\left(-u(\gamma_1)+z\right)}e^{2 \textup{i}c_b(t+ z)(\beta_2-\gamma_2)}dtdz\\=
\left|\int_{\R} \frac{\Phi_b\left(u(\beta_1)-z\right)}{\Phi_b\left(-u(\gamma_1)+z\right)}e^{2 \textup{i}c_b(\beta_2-\gamma_2)z}dz\right|^2.
\end{multline*}
As the complete balancing conditions take the form $\beta_1=\gamma_1$ and $\beta_2=\gamma_2$, we finally obtain
\beq
\tilde W_b(S^3,4_1)=\left|\int_{\R-i0} \frac{\Phi_b\left(-z\right)}{\Phi_b\left(z\right)}dz\right|^2.
\eeq

\subsection{One vertex H-triangulation of $(S^3,5_2)$}
In the graphical notation of \cite{Andersen:2011bt}, let $X$ be given by the diagram \vspace{.2cm}
\begin{figure}[ht!]
\centering
\includegraphics[scale=1]{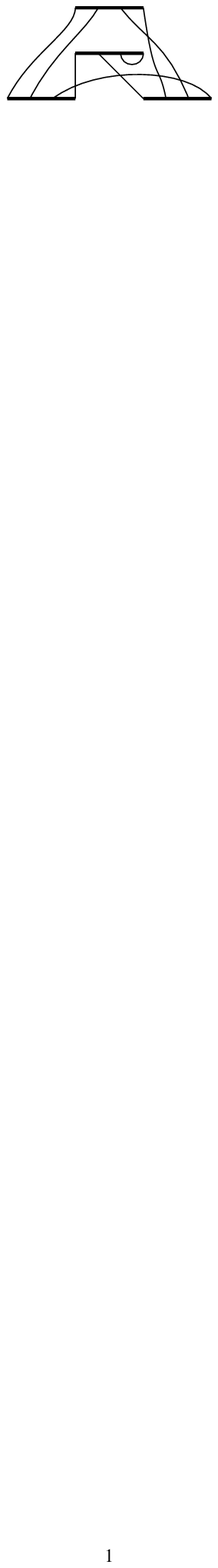}
%\caption{2-3 moves} \label{figure 1}
\end{figure}

One vertex H-triangulation of $(S^3,5_2)$ consists of 4 tetrahedra: $T$ which is a negative tetrahedron and is sitting in the center of the above picture, and $T_1, T_2, T_3$ are positive tetrahedra ($T_1$ is on the left from the central tetrahedron, $T_2$ on the right and $T_3$ is on top). One has to identify the following faces of four tetrahedra:
\beqa
&& \partial_0 T \simeq \partial_1 T, \ \partial_2 T \simeq \partial_0 T_2, \ \partial_3 T \simeq \partial_3 T_1, \ \partial_0 T_1 \simeq \partial_3 T_3, \nonumber \\ && \partial_1 T_1 \simeq \partial_2 T_3, \ \partial_2 T_1 \simeq \partial_3 T_2, \ \partial_1 T_2 \simeq \partial_0 T_3, \ \partial_2 T_2 \simeq \partial_1 T_3.
\eeqa
From the identification of the faces we get the following equalities for the edges
\beqa
&& x_{01} = x_{12}^{(2)} = x_{01}^{(1)} = x_{01}^{(2)} = x_{02}^{(3)} = x_{13}^{(1)}; \nonumber \\ && x_{02} = x_{12} = x_{12}^{(1)} = x_{01}^{(3)} = x_{02}^{(1)}; \nonumber \\ && x_{03} = x_{13} = x_{23}^{(2)} = x_{23}^{(3)} = x_{13}^{(2)}; \nonumber \\ && x_{03}^{(1)} = x_{02}^{(2)} = x_{23}^{(1)} = x_{12}^{(3)} = x_{13}^{(3)} = x_{03}^{(2)} = x_{03}^{(3)},
\eeqa
and just $x_{23}$. Here, the superscripts $(1)$, $(2)$ and $(3)$ refer to tetrahedra $T_i, i=1,2,3$.

The partition function for this one vertex H-triangulation of $(S^3,5_2)$ is equal to the integral of the product of four Boltzmann weights of four tetrahedra
\beqa
&& W_b(S^3,5_2) = \int_{\R^5} \mathcal{B}(\Delta \alpha_1,\Delta \alpha_2 + \textup{i} \nabla (x_1+x_{23}-x_3-x_2)) \nonumber \\ \nonumber && \makebox[0em]{} \times \mathcal{B}(\Delta \beta_1 - \textup{i} \nabla (x_3' - x_1), \Delta \beta_2 - \textup{i} \nabla  (x_1 - x_2)) \\ \nonumber && \makebox[0em]{} \times \mathcal{B}(\Delta \gamma_1 - \textup{i} \nabla  (x_1 - x_3), \Delta \gamma_2 - \textup{i} \nabla  (x_3 - x_3')) \nonumber \\ && \makebox[0em]{} \times \mathcal{B}(\Delta \delta_1 - \textup{i} \nabla (x_3 - x_1), \Delta \delta_2 - \textup{i} \nabla (x_2 + x_3 - 2 x_3')) \delta(x_3'/2) dx_{23} dx_1 dx_2 dx_3 dx_3' \nonumber \\ && \makebox[2em]{} = 2 \tilde W_b(S^3,5_2) \int_\R \mathcal{B}(\Delta \alpha_1,\Delta \alpha_2 + \textup{i} \nabla t) dt \nonumber \\ && \makebox[2em]{} = 2 \tilde W_b(S^3,5_2) \left| \Phi_b(u(\alpha_1)) \right|^2,
\eeqa
where $x_i = x_{0i}, i=1,2,3$ and also we denote $x_3' = x_{03}^{(1)}$.
Here one has
\beqa
&& \nonumber \makebox[-2em]{} \tilde W_b(S^3,5_2) = \int_{\R^3} dx_1 dx_2 dx_3 \mathcal{B}(\Delta \beta_1 + \textup{i} \nabla x_1, \Delta \beta_2 - \textup{i} \nabla (x_1 - x_2)) \nonumber \\ && \makebox[-0.6em]{} \times \mathcal{B}(Q-\Delta \gamma_1-\Delta \gamma_2+ \textup{i} \nabla x_1, \Delta \gamma_2 - \textup{i} \nabla x_3) \mathcal{B}(\Delta \delta_1 + \textup{i} \nabla x_1, \Delta \delta_2 - \textup{i} \nabla (x_2 + x_3)) \nonumber \\ && \makebox[-2.3em]{} = \int_{\R^3} \psi \left( u(\beta_1) + \frac{x_1}{2}, - x_1 + x_2 - 2c_b\frac{\beta_2}{\pi} \right)  \psi \left( u(\pi-\gamma_1-\gamma_2) + \frac{x_1}{2}, - x_3 - 2c_b \frac{\gamma_2}{\pi} \right) \nonumber \\ && \makebox[2em]{} \times \psi \left( u(\delta_1) + \frac{x_1}{2}, - x_2 - x_3 - 2c_b \frac{\delta_2}{\pi} \right) dx_1 dx_2 dx_3,
\eeqa
then writing the definition for $\psi(a,b)$ one gets
\beqa
&& \makebox[-2em]{} \tilde W_b(S^3,5_2) = \int_{\R^6} dx_1 dx_2 dx_3 ds dt du  \nonumber \\ && \makebox[-1.4em]{} \times \frac{\Phi_b(u(\beta_1) + \frac{x_1}{2} + s)\Phi_b(u(\pi-\gamma_1-\gamma_2) + \frac{x_1}{2} + t)\Phi_b(u(\delta_1) + \frac{x_1}{2} + u)}{\Phi_b(-u(\beta_1) - \frac{x_1}{2} + s) \Phi_b(-u(\pi-\gamma_1-\gamma_2) - \frac{x_1}{2} + t)\Phi_b(-u(\delta_1) - \frac{x_1}{2} + u)} \nonumber \\ && \makebox[-1em]{} \times e^{2 \pi \textup{i} (- x_1 + x_2)s - 4 \textup{i} c_b \beta_2 s} e^{-2 \pi \textup{i} x_3 t - 4 \textup{i} c_b \gamma_2 t} e^{-2 \pi \textup{i} (x_2 + x_3)u - 4 \textup{i} c_b \delta_2 u} \nonumber \\ && \makebox[0.5em]{} = \int_{\R^4} dx_1 ds dt du \frac{\Phi_b(u(\beta_1) + \frac{x_1}{2} + s)}{\Phi_b(-u(\beta_1) - \frac{x_1}{2} + s)} \frac{\Phi_b(u(\pi-\gamma_1-\gamma_2) + \frac{x_1}{2} + t)}{\Phi_b(-u(\pi-\gamma_1-\gamma_2) - \frac{x_1}{2} + t)} \nonumber \\ && \makebox[1em]{} \times \frac{\Phi_b(u(\delta_1) + \frac{x_1}{2} + u)}{\Phi_b(-u(\delta_1) - \frac{x_1}{2} + u)} e^{-2 \pi \textup{i} x_1 s - 4 \textup{i} c_b \beta_2 s} e^{- 4 \textup{i} c_b \gamma_2 t} e^{- 4 \textup{i} c_b \delta_2 u} \nonumber \\ && \makebox[2em]{} \times \int_\R dx_2 e^{2 \pi \textup{i} (s-u) x_2} \int_\R dx_3 e^{-2 \pi \textup{i} (t+u) x_3} \nonumber \\ && \makebox[1em]{} = \int_{\R^4} dx_1 ds dt du \frac{\Phi_b(u(\beta_1) + \frac{x_1}{2} + s)}{\Phi_b(-u(\beta_1) - \frac{x_1}{2} + s)} \frac{\Phi_b(u(\pi-\gamma_1-\gamma_2) + \frac{x_1}{2} + t)}{\Phi_b(-u(\pi-\gamma_1-\gamma_2) - \frac{x_1}{2} + t)} \nonumber \\ && \makebox[-1em]{} \times \frac{\Phi_b(u(\delta_1) + \frac{x_1}{2} + u)}{\Phi_b(-u(\delta_1) - \frac{x_1}{2} + u)} e^{-2 \pi \textup{i} x_1 s - 4 \textup{i} c_b \beta_2 s} e^{- 4 \textup{i} c_b \gamma_2 t} e^{- 4 \textup{i} c_b \delta_2 u} \delta(s-u) \delta(-t-u) \nonumber \\ && \makebox[1em]{} = \int_{\R^2} dx_1 du e^{-2 \pi \textup{i} x_1 u - 4 \textup{i} c_b (\beta_2-\gamma_2+\delta_2) u} \nonumber \\ \nonumber && \makebox[-1em]{} \times \frac{\Phi_b(u(\beta_1) + \frac{x_1}{2} + u)\Phi_b(u(\pi-\gamma_1-\gamma_2) + \frac{x_1}{2} - u)\Phi_b(u(\delta_1) + \frac{x_1}{2} + u)}{\Phi_b(-u(\beta_1) - \frac{x_1}{2} + u) \Phi_b(-u(\pi-\gamma_1-\gamma_2) - \frac{x_1}{2} - u) \Phi_b(-u(\delta_1) - \frac{x_1}{2} + u)}
\eeqa

Changing the integration variables
$$
x_1 \rightarrow \frac{x_1}{2} + u, \ \ \ x_2 = u - \frac{x_1}{2},
$$
we get
\beqa
&& \tilde W_b(S^3,5_2) = \int_{\R^2} dx_1 dx_2 e^{-\pi \textup{i} (x_1^2-x_2^2)} e^{-2 \textup{i} c_b (\beta_2-\gamma_2+\delta_2) (x_1+x_2)} \nonumber \\ && \makebox[2em]{} \times \frac{\Phi_b(u(\beta_1) + x_1)}{\Phi_b(-u(\beta_1) + x_2)} \frac{\Phi_b(u(\pi-\gamma_1-\gamma_2) - x_2)}{\Phi_b(-u(\pi-\gamma_1-\gamma_2) - x_1)} \frac{\Phi_b(u(\delta_1) + x_1)}{\Phi_b(-u(\delta_1) + x_2)}.
\eeqa
Finally using the inversion relation (\ref{invers}) and shifting integration variables
$$
x_1 \rightarrow x_1 - u(\delta_1), \ \ \ x_2 \rightarrow x_2 + u(\delta_1),
$$\beqa \label{5_2}
&& \tilde W_b(S^3,5_2) \nonumber \\ \nonumber && \makebox[0em]{} = \int_\R dx_2 \frac{e^{\pi \textup{i} x_2^2} e^{- 2 c_b \textup{i}(\pi- \beta_1 - \beta_2 - \delta_2 + \gamma_2) x_2}}{\Phi_b(x_2) \Phi_b(u(\delta_1)-u(\beta_1) +x_2) \Phi_b(u(\delta_1) - u(\pi-\gamma_1-\gamma_2) + x_2)} \nonumber \\ \nonumber && \makebox[1em]{} \times \int_\R dx_1 e^{-\pi \textup{i} x_1^2} e^{-2 c_b \textup{i}(\pi - \beta_1 - \beta_2 - \delta_2 + \gamma_2) x_2} \\ \nonumber && \makebox[1em]{} \times \Phi_b(x_1) \Phi_b(u(\beta_1)-u(\delta_1) + x_1) \Phi_b(u(\pi-\gamma_1-\gamma_2)-u(\delta_1) + x_1) \nonumber \\ && \makebox[2.5em]{} = \left| \int_\R dx \frac{e^{\pi \textup{i} x^2} e^{- 2 c_b \textup{i}(\pi- \beta_1 - \beta_2 - \delta_2 + \gamma_2) x}}{\Phi_b(x) \Phi_b(\frac{\textup{i}}{2} (\beta_1-\delta_1) +x) \Phi_b(c_b-\frac{\textup{i}}{2} (\gamma_1+\delta_1+\gamma_2) + x)} \right|^2.
\eeqa
In the complete balancing case $
\delta_2 = \beta_3 + \gamma_2, \delta_1 = \gamma_3 = \beta_1,
$
one gets
\beq \nonumber
\left| \int_{\R-\textup{i}0} dy \frac{e^{\pi \textup{i} y^2}}{\Phi_b(y)^3} \right|^2.
\eeq
%where $x' = u(\delta_1) - u(\beta_1)=0$.

\subsection{One vertex H-triangulation of $(S^3,6_1)$}
First of all we start from describing of an H-triangulation of $(S^3,6_1)$.
In the graphical notation of \cite{Andersen:2011bt}, let $X$ be given by the diagram
\begin{figure}[ht!]
\centering
\includegraphics[scale=1]{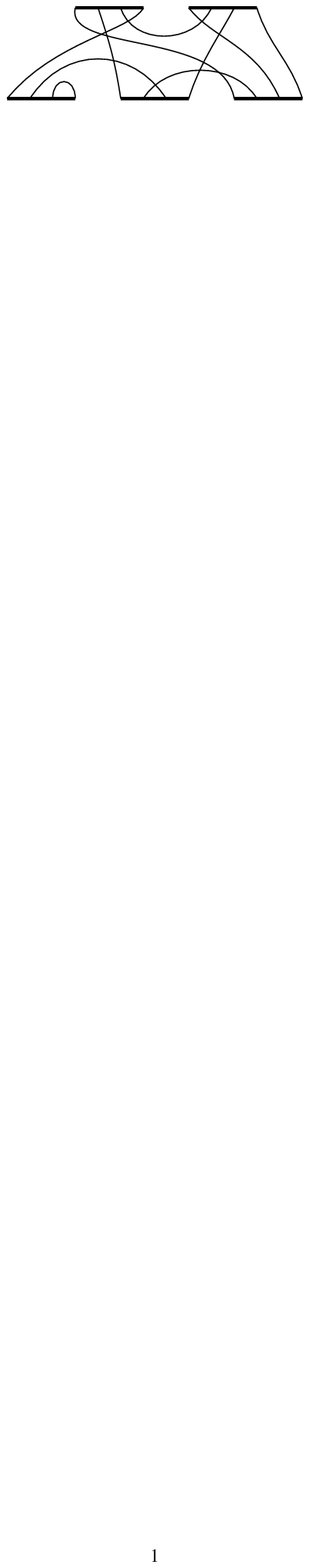}
%\caption{2-3 moves} \label{figure 1}
\end{figure}
%\[
% \begin{tikzpicture}[scale=.3]
% \draw[very thick] (0,0)--(3,0);
% \draw[very thick] (5,0)--(8,0);
% \draw[very thick] (10,0)--(13,0);
% \draw[very thick] (3,4)--(6,4);
% \draw[very thick] (8,4)--(11,4);
% \draw(2,0)..controls (2,1) and (3,1)..(3,0);
% \draw(0,0)..controls (2,5/2) and (5.5,3)..(6,4);
% \draw(1,0)..controls (5/2,7/3) and (5.5,7/3)..(7,0);
% \draw(4,4)..controls (4.5,5/2) and (4.75,5/3)..(5,0);
% \draw(3,4)..controls (5/2,2) and (9.5,3)..(10,0);
% \draw(5,4)..controls (5.5,7/3) and (8.25,7/3)..(9,4);
% \draw(6,0)..controls (7,5/3) and (10,5/3)..(11,0);
% \draw(8,0)..controls (8.5,5/3) and (9.5,3)..(10,4);
% \draw(11,4)..controls (11.5,7/3) and (12.5,5/3)..(13,0);
% \draw(8,4)..controls (9,8/3) and (11,7/3)..(12,0);
% \end{tikzpicture}
%\]

This one vertex H-triangulation of $(S^3,6_1)$ consists of 5 tetrahedra: $T_1$ and $T_3$ which are positive tetrahedra and $T_2, T_4, T_5-$ negative tetrahedra. In the above picture,  tetrahedra $T_1, T_3$ and $T_5$ are situated on the bottom row in the same order from left to right and $T_2, T_4$ lie on the upper row from left to right. According to the picture one has to identify the following faces of five tetrahedra:
\beqa
&& \partial_0 T_1 \simeq \partial_0 T_2, \ \partial_1 T_1 \simeq \partial_2 T_3, \ \partial_2 T_1 \simeq \partial_3 T_1, \ \partial_1 T_2 \simeq \partial_2 T_4, \ \partial_3 T_2 \simeq \partial_0 T_5, \nonumber \\ && \partial_2 T_2 \simeq \partial_0 T_3, \ \partial_1 T_3 \simeq \partial_1 T_5, \ \partial_3 T_3 \simeq \partial_1 T_4, \ \partial_0 T_4 \simeq \partial_3 T_5, \ \partial_3 T_4 \simeq \partial_2 T_5.
\eeqa
From the identification of faces we get the following equalities for the edges
\beqa
&& x_{23}^{(1)} = x_{23}^{(2)} = x_{13}^{(4)} = x_{02}^{(5)} = x_{02}^{(3)} = x_{03}^{(4)} = x_{03}^{(2)} = x_{13}^{(3)}; \nonumber \\ && x_{13}^{(1)} = x_{13}^{(2)} = x_{23}^{(3)} = x_{23}^{(5)} = x_{12}^{(2)} = x_{12}^{(1)}; \nonumber \\ && x_{03}^{(1)} = x_{03}^{(3)} = x_{02}^{(1)} = x_{01}^{(3)} = x_{02}^{(4)} = x_{03}^{(5)}; \nonumber \\ && x_{02}^{(2)} = x_{01}^{(4)} = x_{01}^{(5)} = x_{12}^{(4)} = x_{13}^{(5)}; \nonumber \\ && x_{01}^{(2)} = x_{12}^{(3)} = x_{23}^{(4)} = x_{12}^{(5)},
\eeqa
and just $x_{01}^{(1)}$. Here, the superscripts $(1)$, $(2)$, $(3)$, $(4)$,  and $(5)$ refer to tetrahedra $T_i, i=1,\ldots,5$.

The partition function for this one vertex H-triangulation of $(S^3,6_1)$ is equal to the integral of five Boltzmann weights of five tetrahedra
\beqa
&& W_b(S^3,6_1) = \int_{\R^6} \mathcal{B}(\Delta \alpha_1,\Delta \alpha_2 + \textup{i} \nabla (x_{01}^{(1)} + x_{23}^{(1)} - x_{03}^{(1)} - x_{13}^{(1)})) \delta(x_{23}^{(1)}/2) \nonumber \\ && \makebox[0em]{} \times \mathcal{B}(\Delta \beta_1 +\textup{i} \nabla(x_{23}^{(1)} - x_{02}^{(2)}), \Delta \beta_2 + \textup{i} \nabla (x_{01}^{(2)} - x_{13}^{(1)})) \nonumber \\ \nonumber && \makebox[0em]{} \times \mathcal{B}(\Delta \gamma_1 - \textup{i} \nabla (x_{03}^{(1)} + x_{01}^{(2)} - 2 x_{23}^{(1)}), \Delta \gamma_2 - \textup{i} \nabla (x_{13}^{(1)} - x_{01}^{(2)})) \\ && \makebox[0em]{} \times \mathcal{B}(\Delta \delta_1 + \textup{i} \nabla (x_{03}^{(1)} + x_{01}^{(2)} - x_{23}^{(1)} - x_{02}^{(2)}), \Delta \delta_2 + \textup{i} \nabla (x_{02}^{(2)} + x_{13}^{(1)} - x_{03}^{(1)} - x_{01}^{(2)})) \nonumber \\ \nonumber && \makebox[0em]{} \times \mathcal{B}(\Delta \rho_1 + \textup{i} \nabla  (x_{02}^{(2)} - x_{03}^{(1)}), \Delta \rho_2 + \textup{i} \nabla (x_{01}^{(2)} - x_{23}^{(1)})) dx_{23}^{(1)} dx_{02}^{(2)} dx_{13}^{(1)} dx_{03}^{(1)} dx_{01}^{(2)} \nonumber \\ && \makebox[4.5em]{} = 2 \tilde W_b(S^3,6_1)  \int_\R \mathcal{B}(\Delta \alpha_1,\Delta \alpha_2 + \textup{i} \nabla t) dt \nonumber \\ && \makebox[4.5em]{} = 2 \tilde W_b(S^3,6_1) \left| \Phi_b(u(\alpha_1)) \right|^2,
\eeqa
where
%\beqa
%&& \tilde W_b(S^3,6_1) = \int_{\R^5} \mathcal{B}(\Delta \beta_1 +\textup{i} \nabla(x_{23}^{(1)} - x_{02}^{(2)}), \Delta \beta_2 + \textup{i} \nabla (x_{01}^{(2)} - x_{13}^{(1)})) \nonumber \\ \nonumber && \makebox[1.5em]{} \times \mathcal{B}(\Delta \gamma_1 - \textup{i} \nabla (x_{03}^{(1)} + x_{01}^{(2)} - 2 x_{23}^{(1)}), \Delta \gamma_2 - \textup{i} \nabla (x_{13}^{(1)} - x_{01}^{(2)})) \\ && \makebox[1.5em]{} \times \mathcal{B}(\Delta \delta_1 + \textup{i} \nabla (x_{03}^{(1)} + x_{01}^{(2)} - x_{23}^{(1)} - x_{02}^{(2)}), \Delta \delta_2 + \textup{i} \nabla (x_{02}^{(2)} + x_{13}^{(1)} - x_{03}^{(1)} - x_{01}^{(2)})) \nonumber \\ \nonumber && \makebox[1.5em]{} \times \mathcal{B}(\Delta \rho_1 + \textup{i} \nabla  (x_{02}^{(2)} - x_{03}^{(1)}), \Delta \rho_2 + \textup{i} \nabla (x_{01}^{(2)} - x_{23}^{(1)})) \delta(x_{23}^{(1)}/2) dx_{23}^{(1)} dx_{02}^{(2)} dx_{13}^{(1)} dx_{03}^{(1)} dx_{01}^{(2)}.
%\eeqa
we fixed the edge $x_{23}^{(1)}$ and
\beqa
&& \tilde W_b(S^3,6_1) = \int_{\R^4} \mathcal{B}(\Delta \beta_1 - \textup{i} \nabla x_{02}^{(2)}, \Delta \beta_2 + \textup{i} \nabla (x_{01}^{(2)} - x_{13}^{(1)})) \nonumber \\ && \makebox[1em]{} \times \mathcal{B}(\Delta \gamma_1 - \textup{i} \nabla (x_{03}^{(1)} + x_{01}^{(2)}), \Delta \gamma_2 - \textup{i} \nabla (x_{13}^{(1)} - x_{01}^{(2)})) \nonumber \\ && \makebox[1em]{} \times \mathcal{B}(\Delta \delta_1 + \textup{i} \nabla (x_{03}^{(1)} + x_{01}^{(2)} - x_{02}^{(2)}), \Delta \delta_2 + \textup{i} \nabla (x_{02}^{(2)} + x_{13}^{(1)} - x_{03}^{(1)} - x_{01}^{(2)})) \nonumber \\ && \makebox[1em]{} \times \mathcal{B}(\Delta \rho_1 + \textup{i} \nabla (x_{02}^{(2)} - x_{03}^{(1)}), \Delta \rho_2 + \textup{i} \nabla x_{01}^{(2)}) dx_{02}^{(2)} dx_{13}^{(1)} dx_{03}^{(1)} dx_{01}^{(2)},
\eeqa
now we can change variables $x_{02}^{(2)} \rightarrow x_{02}^{(2)}+x_{03}^{(1)}+x_{01}^{(2)}, x_{03}^{(1)} \rightarrow x_{03}^{(1)} - x_{01}^{(2)}$ and $x_{13}^{(1)} \rightarrow x_{13}^{(1)} + x_{01}^{(2)}$ and one gets
\beqa
&& \tilde W_b(S^3,6_1) = \int_{\R^4} \mathcal{B}(\Delta \beta_1 - \textup{i} \nabla (x_{02}^{(2)} + x_{03}^{(1)}), \Delta \beta_2 - \textup{i} \nabla x_{13}^{(1)}) \\ \nonumber && \makebox[-1.6em]{} \times  \mathcal{B}(\Delta \gamma_1 - \textup{i} \nabla x_{03}^{(1)}, \Delta \gamma_2 - \textup{i} \nabla x_{13}^{(1)}) \mathcal{B}(\Delta \delta_1 - \textup{i} \nabla x_{02}^{(2)}, \Delta \delta_3 - \textup{i} \nabla (x_{01}^{(2)} + x_{13}^{(1)})) \\ \nonumber && \makebox[-1.6em]{} \times  \mathcal{B}(\Delta \rho_1 + \textup{i} \nabla (x_{02}^{(2)} + x_{01}^{(2)}), \Delta \rho_2 + \textup{i} \nabla x_{01}^{(2)}) dx_{02}^{(2)} dx_{13}^{(1)} dx_{03}^{(1)} dx_{01}^{(2)},
\eeqa
where $\delta_3 = \Delta (\pi - \delta_1 - \delta_2)$.
Using now relation (\ref{eq:BPsi}) we rewrite the latter expression as
\beqa
&& \tilde W_b(S^3,6_1) = \int_{\R^4} \psi \left( u(\beta_2) - \frac{x_{13}^{(1)}}{2}, -x_{02}^{(2)} - x_{03}^{(1)} - 2c_b \frac{\beta_1}{\pi} \right) \\ \nonumber && \makebox[-1.5em]{} \times  \psi \left(u(\gamma_2) - \frac{x_{13}^{(1)}}{2}, - x_{03}^{(1)} - 2c_b \frac{\gamma_1}{\pi} \right)  \psi \left( u(\rho_2) + \frac{x_{01}^{(2)}}{2}, x_{02}^{(2)} + x_{01}^{(2)} - 2c_b \frac{\rho_1}{\pi} \right) \\ \nonumber && \makebox[-1.5em]{} \times  \psi \left( u(\delta_3) - \frac{x_{01}^{(2)} + x_{13}^{(1)}}{2}, - x_{02}^{(2)} - 2c_b \frac{\delta_1}{\pi} \right)dx_{02}^{(2)} dx_{13}^{(1)} dx_{03}^{(1)} dx_{01}^{(2)},
\eeqa
and using the definition for $\psi$ function we get
\beqa
&& \tilde W_b (S^3,6_1) = \int_{\R^8} \frac{\Phi_b(u(\beta_2) - \frac{x_{13}^{(1)}}{2}+u)}{\Phi_b(-u(\beta_2) + \frac{x_{13}^{(1)}}{2}+u)} \frac{\Phi_b(u(\gamma_2) - \frac{x_{13}^{(1)}}{2}+v)}{\Phi_b(-u(\gamma_2) + \frac{x_{13}^{(1)}}{2}+v)} \nonumber \\ && \makebox[1em]{} \times \frac{\Phi_b(u(\rho_2) + \frac{x_{01}^{(2)}}{2}+s)}{\Phi_b(-u(\rho_2) - \frac{x_{01}^{(2)}}{2}+s)} \frac{\Phi_b(u(\delta_3) - \frac{x_{01}^{(2)} + x_{13}^{(1)}}{2}+t)}{\Phi_b(-u(\delta_3) + \frac{x_{01}^{(2)} + x_{13}^{(1)}}{2}+t)} \nonumber \\ && \makebox[-1em]{} \times e^{-2 \pi \textup{i} (x_{02}^{(2)} + x_{03}^{(1)})u - 4 \textup{i} c_b \beta_1 u -2 \pi \textup{i} x_{03}^{(1)} v - 4 \textup{i} c_b \gamma_1 v + 2 \pi \textup{i} ( x_{02}^{(2)} + x_{01}^{(2)}) s - 4 \textup{i} c_b \rho_1 s -2 \pi \textup{i} x_{02}^{(2)} t - 4 \textup{i} c_b \delta_1 t} \nonumber \\ && \makebox[1em]{} \times dx_{13}^{(1)} dx_{01}^{(2)} dx_{02}^{(2)} dx_{03}^{(1)} du dv ds dt,
\eeqa
where we can take the integral over $x_{02}^{(2)}$ and $x_{03}^{(1)}$ which produce two delta functions $\delta(-u+s-t)$ and $\delta(-v-u)$ respectively
\beqa
&& \tilde W_b (S^3,6_1) = \int_{\R^6} \frac{\Phi_b(u(\beta_2) - \frac{x_{13}^{(1)}}{2}+u)}{\Phi_b(-u(\beta_2) + \frac{x_{13}^{(1)}}{2}+u)} \frac{\Phi_b(u(\gamma_2) - \frac{x_{13}^{(1)}}{2}+v)}{\Phi_b(-u(\gamma_2) + \frac{x_{13}^{(1)}}{2}+v)} \nonumber \\ && \makebox[1em]{} \times \frac{\Phi_b(u(\rho_2) + \frac{x_{01}^{(2)}}{2}+s)}{\Phi_b(-u(\rho_2) - \frac{x_{01}^{(2)}}{2}+s)} \frac{\Phi_b(u(\delta_3) - \frac{x_{01}^{(2)} + x_{13}^{(1)}}{2}+t)}{\Phi_b(-u(\delta_3) + \frac{x_{01}^{(2)} + x_{13}^{(1)}}{2}+t)} \nonumber \\ && \makebox[1em]{} \times e^{- 4 \textup{i} c_b \beta_1 u}  e^{- 4 \textup{i} c_b \gamma_1 v}  e^{2 \pi \textup{i} x_{01}^{(2)} s - 4 \textup{i} c_b \rho_1 s}  e^{- 4 \textup{i} c_b \delta_1 t} \delta(-u+s-t) \delta(-v-u) \nonumber \\ && \makebox[1em]{} \times dx_{13}^{(1)} dx_{01}^{(2)} du dv ds dt,
\eeqa
taking the integrals over $t$ and $v$ one gets
\beqa
&& \tilde W_b (S^3,6_1) = \int_{\R^4} \frac{\Phi_b(u(\beta_2) - \frac{x_{13}^{(1)}}{2}+u)}{\Phi_b(-u(\beta_2) + \frac{x_{13}^{(1)}}{2}+u)} \frac{\Phi_b(u(\gamma_2) - \frac{x_{13}^{(1)}}{2}-u)}{\Phi_b(-u(\gamma_2) + \frac{x_{13}^{(1)}}{2}-u)} \nonumber \\ && \makebox[1em]{} \times \frac{\Phi_b(u(\rho_2) + \frac{x_{01}^{(2)}}{2}+s)}{\Phi_b(-u(\rho_2) - \frac{x_{01}^{(2)}}{2}+s)} \frac{\Phi_b(u(\delta_3) - \frac{x_{01}^{(2)} + x_{13}^{(1)}}{2}+s-u)}{\Phi_b(-u(\delta_3) + \frac{x_{01}^{(2)} + x_{13}^{(1)}}{2}+s-u)} \nonumber \\ && \makebox[1em]{} \times e^{- 4 \textup{i} c_b (\beta_1-\gamma_1-\delta_1) u} e^{2 \pi \textup{i} x_{01}^{(2)} s - 4 \textup{i} c_b (\rho_1+\delta_1) s} dx_{13}^{(1)} dx_{01}^{(2)} du ds.
\eeqa

Let us consider reparametrization $x=u-\frac12 x_{13}^{(1)}, y=u+\frac12 x_{13}^{(1)}, z=s+\frac12 x_{01}^{(2)}, w=s-\frac12 x_{01}^{(2)}$
\beqa
&& \tilde W_b (S^3,6_1) = \int_{\R^4} \frac{\Phi_b(u(\beta_2) +x)}{\Phi_b(-u(\beta_2) +y)} \frac{\Phi_b(u(\gamma_2) - y)}{\Phi_b(-u(\gamma_2) - x)} \nonumber \\ && \makebox[1em]{} \times \frac{\Phi_b(u(\rho_2) + z)}{\Phi_b(-u(\rho_2) +w)} \frac{\Phi_b(u(\delta_3) +w-y)}{\Phi_b(-u(\delta_3) + z-x)} \nonumber \\ && \makebox[1em]{} \times e^{-2 \textup{i} c_b (\beta_1-\gamma_1-\delta_1) (x+y)} e^{\pi \textup{i} (z^2-w^2)} e^{- 2 \textup{i} c_b (\rho_1+\delta_1) (z+w)} dx_{13}^{(1)} dx_{01}^{(2)} du ds.
\eeqa
which is equal to
\beqa \label{Knot6}
&& \makebox[2em]{} \tilde W_b(S^3,6_1) \\ \nonumber && \makebox[-2em]{} = \left| \int_{\R^2} \frac{\Phi_b(u(\beta_2)+x)\Phi_b(u(\rho_2)+z)}{\Phi_b(-u(\gamma_2)-x)\Phi_b(-u(\delta_3)+z-x)} e^{-2 \textup{i} c_b \left( (\beta_1-\gamma_1-\delta_1) x + (\rho_1 + \delta_1) z \right) + \pi \textup{i} z^2} dx dz \right|^2,
\eeqa
demonstrating the factorization for the H-triangulation of $(S^3,6_1)$ and which is consistent with Conjecture~\ref{conj}.

In the complete balancing case one has
$$
\beta_2 = \gamma_2, \ \beta_2 = \rho_2 + \delta_3, \ \beta_1 = \rho_1 + \delta_1, \ \rho_2 = \rho_1 + \delta_1 - 1,
$$
simplifying (\ref{Knot6}) to the following expression after shifting $x \rightarrow x-u(\beta_2), z \rightarrow z-u(\rho_2)$
\beqa
&& \makebox[2em]{} \tilde W_b(S^3,6_1) = \left| \int_{\R^2} \frac{\Phi_b(x)\Phi_b(z)}{\Phi_b(-x)\Phi_b(-c_b+z-x)} e^{-4 \pi \textup{i} c_b z + \pi \textup{i} z^2} dx dz \right|^2,
\eeqa

\section{Application to $3d$ supersymmetric field theories}\label{3dSUSY}
\subsection{$3d$ supersymmetric theories living on a squashed three-sphere}
Following the work of Pestun~\cite{Pestun}, the partition functions of $3d$ $\mathcal{N}=2$ supersymmetric theories, defined on a squashed three-sphere $S_b^3$, were calculated in the papers \cite{KWY,Jafferis,Hama} by using the localization method. These partition functions are given in the form of integrals with the integrands composed of hyperbolic gamma functions \cite{Willett,DSV}. For any $3d$ $\mathcal{N}=2$ supersymmetric theory defined on $S_b^3$ with a gauge group $G$ and a flavour group $F$, the corresponding partition function has the following structure
\beq \label{PF_def}
Z(\underline{f}) \ = \ \int_{-\textup{i}\infty}^{\textup{i}\infty}
 \prod_{j=1}^{\textup{rank} G}du_j\,  J(\underline{u})
Z^{vec}(\underline{u}) \prod_{I} Z_{\Phi_I}^{chir}(\underline{f},\underline{u}).
\eeq
Here the integral is taken over $u_j$-variables which are associated with the Weyl weights for the Cartan
subalgebra of the gauge group $G$ and the $f_k$'s denote the chemical potentials for the flavor symmetry group $F$\footnote{From physical point of view, $f_k$'s are linear combinations of the $R$-charge, the masses of the hypermultiplets, and the Fayet--Illiopoulos terms associated to the additional Abelian global symmetries.}. For CS theory one has
$J(\underline{u})=e^{-\pi \textup{i} k \sum_{j=1}^{\textup{rank} G} u_j^2}$,
where $k$ is the level of the CS-term, while for SYM theories one has
$J(\underline{u}) = e^{2 \pi \textup{i} \lambda \sum_{j=1}^{\textup{rank} G} u_j}$,
where $\lambda$ is the Fayet--Illiopoulos term. The terms
$Z^{vec}(\underline{u})$ and $Z_{\Phi_I}^{chir}(\underline{f},\underline{u})$
 in (\ref{PF_def}) come from the vector superfield and the matter fields,
repectively, and are given in terms of the hyperbolic gamma function.

The result of localization allows us to relate the physical theory with some matrix integral of the form (\ref{PF_def}). Also we can invert the logic: having some matrix integral of the type (\ref{PF_def}) one can find a $3d$ $\mathcal{N}=2$ supersymmetric field theory whose partition function is given by this matrix integral \cite{DSV}. Thus, all the partition functions which we get by considering (\ref{eq:pf}) can be interpreted as partition functions for some $3d$ $\mathcal{N}=2$ supersymmetric field theories. Moreover, as the expression (\ref{eq:pf}) corresponds to some triangulation of a $3$-dimensional manifold $M$,  we obtain a link between $3$- manifolds and $3d$ $\mathcal{N}=2$ supersymmetric field theories defined on $S_b^3$. This is known as a $3d/3d$ duality considered recently in \cite{Terashima:2011qi,Dimofte:2011ju,Dimofte:2011py,Teschner:2012em} (see also \cite{SV} for the relation of the objects to four-dimensional supersymmetric field theories).

In \cite{Dimofte:2011ju}, the state variables live in the faces, while in our case   the state variables live on the edges. To get a $3d$ theory from $3d$ manifold $M$ one has to triangulate this manifold and calculate its partition function (\ref{eq:pf}) and then interpret this expression as a partition function (\ref{PF_def}). One should notice that every common edge corresponds to abelian gauge group.

%From the above formulas we observe that all our partition functions constructed for three-dimensional manifolds have also interpretation in terms of some partition functions of some $3d$ $\mathcal{N}=2$ supersymmetric theories living on a squashed three-sphere $S_b^3$. Different representations of $3d$ manifolds in terms of tetrahedra correspond to different physical theories related by electro-magnetic duality.

Let us start from our building block: the tetrahedral Boltzmann weight composed of three hyperbolic gamma functions, each corresponding to the contribution coming from the $3d$ $\mathcal{N}=2$ chiral hypermultiplet. Namely,
\beq
\mathcal{B}(T,x) = \prod_{i=1}^3 \gamma^{(2)}(\Delta \alpha_i + \textup{i} \nabla (x_{i+1}+x_{i+1}'-x_{i-1}-x_{i-1}');\omega_1,\omega_2),
\eeq
corresponds to three chiral superfields $Q_i,i=1,2,3$, with $SU(3)$ global symmetry group (since $\sum_{i=1}^3 \textup{i}(x_{i+1}+x_{i+1}'-x_{i-1}-x_{i-1}') = 0$) and a superpotential
$$W \sim Q_1 Q_2 Q_3,$$
which has a correct $R$-charge. This can be easily seen from the fact that the dihedral angles $\alpha_i, i=1,2,3$ correspond to $R$-charges of three chiral superfields. And since $\sum_{i=1}^3 \alpha_i = \pi$ then the $R_W$charge of the superpotential $W$ is given as $R_W = \sum_{i=1}^3 R_{Q_i} = \sum_{i=1}^3 2 \alpha_i/\pi  = 2$.

%\begin{figure}[ht]\vspace{0.3cm}
%\begin{center}
%\leavevmode \epsfxsize=4cm \epsffile{Pachner.eps}
%\end{center}\vspace{0.2cm}
%\caption{$2-3$ Pachner move.}\label{Pachner}
%\end{figure}

The first non-trivial case of a $3d$ theory with a non-trivial gauge group is the pentagon identity (\ref{Pent1}) (once again we stress that it is a $2-3$ Pachner move) when we take two positive tetrahedra and glue them together over the common face. The partition function of two glued tetrahedra having vertices $(0,1,2,4)$ and $(0,2,3,4)$ %depicted in Fig. \ref{Pachner}
is
\beqa \label{PentL}
&& \makebox[-0.3em]{} W_{b,A} = \mathcal{B}(\Delta \alpha_1 + \textup{i} \nabla (x_{02}+x_{34}-x_{03}-x_{24}), \Delta \alpha_2 + \textup{i} \nabla (x_{03}+x_{24}-x_{04}-x_{23})) \nonumber \\ && \makebox[1.3em]{} \times \mathcal{B}(\Delta \beta_1 + \textup{i} \nabla (x_{01}+x_{24}-x_{02}-x_{14}), \Delta \beta_2 + \textup{i} \nabla (x_{02}+x_{14}-x_{04}-x_{12})),
\eeqa
where $\sum_{i=1}^3 (a_i+b_i)=\omega_1+\omega_2$ which is the partition function for a theory $A$ which consists of six $3d$ $\mathcal{N}=2$ free chiral hypermultiplets with $F=SU(3) \times SU(3) \times U(1)$ global symmetry group. Here we have $\alpha_1 = a_2+b_1, \alpha_2=a_3+b_2, \beta_1 = a_1+b_2$ and $\beta_2 = a_3 + b_1$. Here each $SU(3)$ corresponds to separate tetrahedron and $U(1)$ group distinguishes the two tetrahedra. At the same time, using $2-3$ Pachner move, two glued tetrahedra can be considered as three tetrahedra with the vertices $(0,1,2,3), (0,1,3,4)$ and $(1,2,3,4)$ having a common edge $x_{04}$ whose partition function is
\beqa \label{PentR}
&& \makebox[-1.5em]{} W_{b,B} = \int_\R \mathcal{B}(\Delta a_1 + \textup{i} \nabla (x_{01}+x_{23}-x_{02}-x_{13}),\Delta b_1 + \textup{i} \nabla (x_{02}+x_{13}-x_{03}-x_{12})) \nonumber \\ && \makebox[1em]{} \times \mathcal{B}(\Delta a_2 + \textup{i} \nabla (x_{12}+x_{34}-x_{13}-x_{24}),\Delta b_2 + \textup{i} \nabla (x_{13}+x_{24}-x_{14}-x_{23})) \\ \nonumber && \times \mathcal{B}(\Delta a_3 + \textup{i} \nabla (x_{01}+x_{34}-x_{03}-x_{14}),\Delta b_3 + \textup{i} \nabla (x_{03}+x_{14}-x_{04}-x_{13})) dx_{13},
\eeqa
where $\sum_{i=1}^3 (a_i+b_i) = \omega_1+\omega_2$.

Expression (\ref{PentR}) gives a partition function for $3d$ $\mathcal{N}=2$ SQED theory $B$ (which has $U(1)$ gauge group) with $3$ flavors and overall $F=SU(3) \times SU(3) \times U(1)$ global symmetry group and $2$ singlet baryons. There are three tetrahedra in this picture so one can think of $SU(3)^3$ global symmetry group but the part of this, namely, $U(1)$ becomes a gauge group leaving $SU(3) \times SU(3) \times U(1)$ global symmetry group.

Since (\ref{PentL})=(\ref{PentR}) the partition functions for theories $A$ and $B$ are the same which suggests the duality between these theories. Generally, different triangulations of $3$-manifolds produce different phases of the same theory, in other words, we get dual descriptions for $3d$ supersymemtric field theories related to a given $3$-dimensional manifold.

%\begin{figure}[ht]\vspace{0.3cm}
%\begin{center}
%\leavevmode \epsfxsize=5cm \epsffile{octachedron.eps}
%\end{center}\vspace{-0.2cm}
%\caption{Octahedron built from four tetrahedra.}\label{Octahedron}
%\end{figure}

One can continue further and construct triangulations for other $3$-manifolds and relate them to $3d$ supersymmetric field theories. As a next example, we consider four tetrahedra built from vertices $(0,1,2,5), (0,2,3,5), (0,3,4,5), (0,1,4,5)$ glued together over a common edge $x_{05}$ to form an  octahedron. We have four tetrahedra: three positive $T_1, T_2, T_3$ and one negative $T_4$. These tetrahedra have the following vertices: $T_1 = \{0,1,2,5\}, T_2 = \{0,2,3,5\}, T_3 = \{0,3,4,5\}, T_4 = \{0,1,4,5\}$. %From Fig. \ref{Octahedron} we have the following identification of faces
Identifying the faces, we get
\beq
\partial_1 T_1 \simeq \partial_3 T_2, \ \ \ \ \partial_2 T_2 \simeq \partial_4 T_3, \ \ \ \ \partial_3 T_3 \simeq \partial_1 T_4, \ \ \ \ \partial_2 T_1 \simeq \partial_4 T_4, \ \ \ \
\eeq
from which we get
\beqa
&& x_{05}^{(1)} = x_{05}^{(2)} = x_{05}^{(3)} = x_{05}^{(4)}, \ \ x_{25}^{(1)} = x_{25}^{(2)}, \ \  x_{02}^{(1)} = x_{02}^{(2)}, \ \ x_{35}^{(2)} = x_{35}^{(3)}, \nonumber \\ && x_{03}^{(2)} = x_{03}^{(3)}, \ \ x_{45}^{(3)} = x_{45}^{(4)}, \ \ x_{04}^{(3)} = x_{04}^{(4)}, \ \ x_{15}^{(1)} = x_{15}^{(4)}, \ \ x_{01}^{(1)} = x_{01}^{(4)}.
\eeqa
so that the partition function is equal to
\beqa
&& \makebox[1em]{} W_{b,\text{Octahedron}} = \int d x_{05}^{(1)} \nonumber \\ && \makebox[-0.3em]{} \times \mathcal{B}(\Delta \alpha_1 + \textup{i} \nabla (x_{02}^{(1)}+x_{15}^{(1)}-x_{12}^{(1)}-x_{05}^{(1)}), \Delta \alpha_2 + \textup{i} \nabla (x_{12}^{(1)}-x_{01}^{(1)}-x_{25}^{(1)}+x_{05}^{(1)})) \nonumber \\ && \makebox[-0.3em]{} \times  \mathcal{B}(\Delta \beta_1+\textup{i} \nabla (x_{03}^{(2)}+x_{25}^{(1)}-x_{23}^{(2)}-x_{05}^{(1)}), \Delta \beta_2+\textup{i} \nabla (x_{23}^{(2)}-x_{02}^{(1)}-x_{35}^{(2)}+x_{05}^{(1)})) \nonumber \\ && \makebox[-0.3em]{} \times \mathcal{B}(\Delta \gamma_1+\textup{i} \nabla (x_{04}^{(3)}+x_{35}^{(2)}-x_{34}^{(3)}-x_{05}^{(1)}), \Delta \gamma_2+\textup{i} \nabla (x_{34}^{(3)}-x_{03}^{(2)}-x_{45}^{(3)}+x_{05}^{(1)})) \nonumber \\ && \makebox[-0.3em]{} \times \mathcal{B}(\Delta \delta_1 - \textup{i} \nabla (x_{14}^{(4)}-x_{01}^{(1)}-x_{45}^{(3)}+x_{05}^{(1)}), \Delta \delta_2-\textup{i} \nabla (x_{04}^{(3)}+x_{15}^{(1)}-x_{14}^{(4)}-x_{05}^{(1)})),
\eeqa
which corresponds to the partition function of $3d$ $\mathcal{N}=2$ SQED theory with $4$ flavors and four singlet baryons with the overall global symmetry $SU(3)^3 \times U(1)$. Octahedron can be also represented by gluing five tetrahedra which is not so obvious from geometrical point of view and will be much easier to see from the next subsection using the Bailey tree technique. As we will show the triangulation with five tetrahedra gives a dual description of the starting theory in terms of a quiver gauge theory with $U(1) \times U(1)$ gauge group.

Continuing further and gluing more tetrahera one gets a class of $3d$ supersymmetric field theories corresponding to a given triangulation of a $3$-manifold. For example, gluing $F$ tetrahedra along one common edge one gets the partition function for $3d$ $\mathcal{N}=2$ SQED theory with $F$ flavors and $F$ additional singlet baryons which has $SU(3)^{F-1} \times U(1)$ global symmetry group (since $U(1)$ becomes a gauge group).

\subsection{Bailey tree technique}
There is an alternative way to see the results of the previous subsection based on the application of Bailey tree technique for hyperbolic integrals (very much in the spirit of \cite{Spir:2003}). This approach  gives an algebraic way of getting the partition functions and relates  the triangulated $3$-dimensional manifolds from one hand side, and $3d$ supersymmetric field theories defined on a squashed three sphere, from the other. The Bailey tree technique is useful for tracking different triangulations related to each other by $2-3$ Pachner move from the algebraic viewpoint.

\begin{definition}We say that two functions $\alpha(z,t)$ and $\beta(z,t), z,t \in \mathbb{C}$ form an integral hyperbolic Bailey pair (the hyperbolic level) with respect to the parameter $t$ if
\beq \label{BaileyP}
\beta(w,t) = \int \mathcal{B}(t+w-z,t-w+z) \alpha(z,t) dz.
\eeq
\end{definition}

\begin{theorem} [follows from Theorem 1 \cite{Spir:2003}]\label{thm2}
Whenever two functions $\alpha(z, t)$ and $\beta(z, t)$ form an integral hyperbolic
Bailey pair with respect to $t$, the new functions
\beq
\alpha'(w,s+t) = \mathcal{B}(t+u+w,2s) \alpha(w,t)
\eeq
and
\beq
\beta'(w,s+t) = \int \mathcal{B}(s+w-x,u+x) \mathcal{B}(s+2t+u+w,s-w+x) \beta(x,t) dx,
\eeq
form an integral hyperbolic Bailey pair with respect to parameter $s+t$.
\end{theorem}
\begin{proof} The proof is similar to the proof in the elliptic case \cite{Spir:2003}. We start from the definition for $\beta'(w,s+t)$:
\beq \nonumber
\beta'(w,s+t) = \int \mathcal{B}(s+w-x,u+x) \mathcal{B}(s+2t+u+w,s-w+x) \beta(x,t) dx,
\eeq
where we substitute $\beta(x,t)$ from equation~\eqref{BaileyP}:
\beqa \label{Oktah5}
&& \beta'(w,s+t) = \int \mathcal{B}(s+w-x,u+x) \mathcal{B}(s+2t+u+w,s-w+x) \nonumber \\ && \makebox[6em]{} \times \mathcal{B}(t+x-y,t-x+y) \alpha(y,t) dy dx.
\eeqa
In the latter expression we can apply formula
\beq \label{EllBInt}
\int \prod_{i=1}^3 \gamma^{(2)}(a_i -u;\omega_1,\omega_2) \gamma^{(2)}(b_i+u;\omega_1,\omega_2) du = \prod_{i,j=1}^3 \gamma^{(2)}(a_i+b_j;\omega_1,\omega_2),
\eeq
where $\sum_{i=1}^3 (a_i+b_i) = \omega_1+\omega_2$, so that we get
\beq \label{Oktah}
\beta'(w,s+t) = \int \mathcal{B}(s+t+w-x,s+t-w+x) \alpha'(x,s+t) dx.
\eeq
\end{proof}
From identity~(\ref{EllBInt}) one gets the following Bailey pair
\beq
\alpha(z,t) = \prod_{i=1}^2 \mathcal{B}(\alpha_i-z,\beta_i+z),
\eeq
where $2t + \sum_{i=1}^2 (\alpha_i + \beta_i)=\omega_1+\omega_2$, and
\beq
\beta(w,t) = \prod_{i=1}^2 \mathcal{B}(t+w+\alpha_i,t-w+\beta_{3-i}).
\eeq
The pentagon identity permits us to define a particular Bailey pair thus giving the definition for Bailey pairs a topological interpretation in terms of the 2-3 Pachner move. In other words, the construction of new Bailey pairs through Theorem~\ref{thm2} corresponds to changing a triangulation by the 2-3 Pachner move.

In  the case of an octahedron triangulated into four tetrahedra which we considered in the previous subsection, the partition function (\ref{Oktah}) can be written as
\beq \label{4Delta}
Z_{4\Delta's} = \int \mathcal{B}(s+t+w-x,s+t-w+x) \mathcal{B}(t+u+x,2s) \prod_{i=1}^2 \mathcal{B}(\alpha_i-x,\beta_i+x) dx,
\eeq
where $2t + \sum_{i=1}^2 (\alpha_i + \beta_i)=\omega_1+\omega_2$. On the other hand, this expression is equal to (\ref{Oktah5})
\beqa
&& Z_{5\Delta's} = \int \mathcal{B}(s+w-x,u+x) \mathcal{B}(s-w+x,2t+s+u+w) \nonumber \\ && \makebox[3.5em]{} \times \mathcal{B}(t+x-y,t-x+y) \prod_{i=1}^2 \mathcal{B}(\alpha_i-y,\beta_i+y) dy dx,
\eeqa
which corresponds to triangulation of the octahedron in terms of five tetrahedra.
Repeating  this procedure, one can further increase the number of tetrahedra thus obtaining new equalities for dual $3d$ supersymmetric field theories related to these triangulations.

\section{Relationship to representation to $PSL(2,\mathbf C)$}

In this section, we briefly discuss the relationship between the  invariant $W_b(X, s,
\lambda)$ and representations of the corresponding fundamental group into $PSL(2,
\mathbf C)$ and simplicial Chern--Simons theory.

\subsection{Angle structures and representations of fundamental groups to $PSL(2, \mathbf C)$}
For simplicity, let us assume that $X=(M, \mathcal T)$ is an
oriented triangulated closed pseudo 3-manifold, i.e., $\partial X
=\emptyset$, where $M$ is the underlying pseudo 3-manifold and
$\mathcal T$ is the triangulation.  Let $\Box (\mathcal T)$ be the
set of all quads in $\mathcal T$. Recall that $\mathbf {Z}/\mathbf
{3Z} = \{1, \tau, \tau^2\}$ acts on $\Box (\mathcal T)$
corresponding to the cyclic order of three edges around each
vertex. For $q \in \Box(\mathcal T)$, we will use $q'$ and $q''$
to denote $\tau(q)$ and $\tau^2(q)$ below.   A \it shaped
structure \rm on $X$ (or $\mathcal T)$ is a function $\alpha:
\Box(\mathcal T) \to (0, \pi)$ so that
$\alpha(q)+\alpha(q')+\alpha(q'')=\pi$ for all $q \in
\Box(\mathcal T)$. The \it weight \rm of a shape structure
$\alpha$ is the function $f:\Delta_1(\mathcal T) \to \mathbf R$
sending each edge $e$ to $f(e) =\sum_{q \sim e} \alpha(q)$ where
$q \sim e$ means the quad $q$ faces the edge $e$.  In particular,
an \it angle structure \rm is a shaped structure whose weight at
each edge is $2\pi$.

The invariant  $W_b(X)$ in Theorem~\ref{main} is defined for each shaped
triangulation, i.e., $W_b(X)$ $= $$W_b(M; \mathcal T, \alpha)$.
Theorem~\ref{main} implies that $W_b(M; \mathcal T_3, \alpha) =W_b(M;
\mathcal T_2, \beta)$ if $\mathcal T_3$ is obtained from $\mathcal
T_2$ by a 2-3 Pachner move so that $\beta$ is the angle structure
on $\mathcal T_2$ induced by the angle structure $\alpha$ on
$\mathcal T_3$. (The equation for defining $\beta$ from $\alpha$
is indicated in figure 1.) In general, there are many different
(or may be none) angle structures on $\mathcal T_3$ inducing the
same angle structure on $\mathcal T_2$.  These different angles
structures are related by a gauge transformation induced by the
degree 3 edge in $\mathcal T_3$.  Theorem~\ref{main} says that $W_b(M;
\mathcal T, \alpha)$ depends only on the (edge type) gauge
equivalence class of the shaped structure $\alpha$.

We will describe briefly the edge type gauge equivalence class
now. Recall that a \it tangential angle structure \rm on $\mathcal
T$ (see \cite{Luo}) is a map $x: \Box(\mathcal T) \to \mathbf R$
so that for each $q \in \Box(\mathcal T)$, $x(q)+ x(q') +
x(q'')=0$ and for each edge $e \in  \Delta_1(\mathcal T)$,
$\sum_{q \sim e} x(q)=0$. Thus the space of all tangential angle
structures is a vector space, denoted by $TAS(\mathcal T)$. For
any shape structure $\alpha$, $v \in TAS(\mathcal T)$ and small
$t$, $\beta=\alpha + tv$ is still a shape structure so that
$\beta$ and $\alpha$ have the same weight.
 A generating set
of vectors for $TAS(\mathcal T)$ was well known and can be
described as follows.  Consider the vertex link $lk(v)$ of a
vertex $v \in \Delta_0(\mathcal T)$.
  Let $s$ be an edge loop in the dual
CW-decomposition of the triangulated surface $lk(v)$. The loop $s$
can be described as a sequence of triangles $\{t_1, ..., t_n\}$
and edges $\{\epsilon_1, ... , \epsilon_n\}$ in $lk(v)$ so that
$\epsilon_i$ is adjacent to $t_i$ and $t_{i+1}$ ($t_{n+1}=t_1$).
Since each $t_i$ corresponds to a tetrahedron $T_i$ and each
$\epsilon_i$ corresponds to a co-dimension-1 face $F_i$ in
$\mathcal T$, each $T_i$ contains a unique quad $q_i$ facing the
edge $F_{i-1} \cap F_{i}$ in $T_i$.
%We will denote the edge loop $s$ by $\{q_1, ..., q_n\}$.

%\medskip
%\epsfxsize=3.3truein \centerline{\epsfbox{2.eps}}
%\centerline{figure 2} \vskip 4mm

\begin{figure}[ht!]
\centering
\includegraphics[scale=0.55]{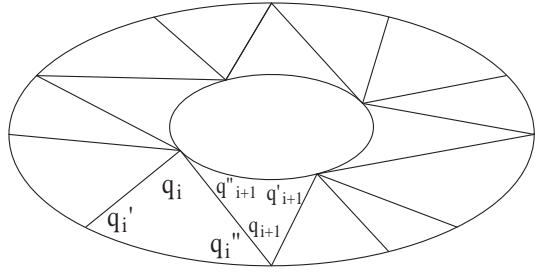}
\caption{Edge loop in a vertex link} \label{figure 2}
\end{figure}
Define a map $g_s: \Box(\mathcal T) \to \mathbf R$ by
$g_s(q_i')=1$, $g_s(q_i'')=-1$ and $g_s(q)=0$ for all other $q$'s.
One checks easily that $g_s \in TAS(\mathcal T)$. In particular,
if $s$ is the loop around a vertex $u$ in $lk(v)$, then $g_s$ is
the gauge transformation associated to the edge $e$ corresponding
to $u$. Two shaped structures $\alpha$ and $\beta$ on $\mathcal T$
are \it edge type gauge equivalent \rm if their difference $\alpha
-\beta$ is a linear combinations of $g_s$'s for edge loops $s$
which are around vertices in vertex links.  Theorem 1 says that
$W_b(M; \mathcal T, \alpha)$ depends only on the edge type gauge
equivalence classes. A theorem in \cite{Til} shows $TAS(\mathcal
T)$ is generated by vectors $g_s$.
%Our calculation shows that there are shaped
%structures $\alpha$ and some edge loop $s$ in $lk(v)$ so that
%$W_b(M; \mathcal T, \alpha)$ is not the same as $W_b(M; \mathcal
%T, \alpha + tg_s)$ for small $t$.
Define the \it angle holonomy \rm $\alpha(s)$ of a shaped
structure $\alpha$ along an edge loop $s$ in $lk(v)$ to be
$\sum_{i=1}^n \alpha(q_i)$. The work of \cite{Til} and
\cite{Andersen:2011bt} show two shaped structures are edge type
gauge equivalent if and only if they have the same angle holonomy
along any edge loop $s$ in vertex links.   This suggests a way to
represent the edge type gauge equivalence class of shaped
structures using volume optimization. Namely, given a shaped
structure $\alpha$, let $A_{\alpha}$ be the set of all shaped
structures on $\mathcal T$ edge type gauge equivalent to $\alpha$.
The \it volume \rm of a shape structure is the sum of the volume
of the hyperbolic tetrahedra determined by the shape. It is well
known that volume is a strictly concave function of shape
structure $\alpha$. In particular, there is at most one shape
structure $\beta \in A_{\alpha}$ which has the maximum volume.
Note that it may not exist in $A_{\alpha}$, i.e., the maximum
volume point may appear in the boundary of the closure of
$A_{\alpha}$. Suppose now that $\alpha$ is an angle structure and
the maximum volume $\beta$ exists in $A_{\alpha}$. Then by the
standard volume optimization method (see \cite{Riv}, \cite{FG}, or
\cite{Luo}), one sees that the complex shape parameter $z_{\beta}$
given by  (\ref{shape-p}) associated to $\beta$ satisfies
Thurston's gluing equation. Therefore, it produces a
representation $\rho$ of $\pi_1(M-\Delta_0(\mathcal T))$ to
$PSL(2, \mathbf C)$ so that for any edge loop $s$ in $lk(v)$, the
eigenvalues of $\rho(s)$ are of the form $r e^{\pm \sqrt{-1}
\beta(s)/2}$ for $r \in \mathbf R_{>0}$. This shows if there
exists an angle structure of the maximum volume edge type gauge
equivalent to $\alpha$, one can assigns the invariant $W_b(M;
\mathcal T, \alpha)$ to the representation $\rho$, i.e., the
invariant $W_b(M; \mathcal T, \alpha)$ may be an invariant of a
pair $(M, \rho)$. The precise conjectural picture of $W_b(M;
\mathcal T, \alpha)$ is: if two angle structures $(\mathcal T_i,
\alpha_i)$ ($i=1,2$) are associated to the same representation
$\rho$, then $W_b(M; \mathcal T_1, \alpha_1) =W_b(M; \mathcal T_2,
\alpha_2)$.

\subsection{Relationship with Simplicial $PSL(2,\mathbf R)$ Chern--Simons
theory}

In \cite{Luo1}, we proposed a variational principle for finding
real valued solutions of Thurston's equation on a triangulated
oriented closed pseudo 3-manifold $(M; \mathcal T)$.  Given $(M;
\mathcal T)$, we introduce the homogeneous Thurston's equation
(HTE) as follows. A map $x: \Box(\mathcal T) \to \mathbf R$ is
said to solve HTE if for each $q \in \Box(\mathcal T)$, $x(q)+
x(q') + x(q'') =0$ and for each edge $e$ in $\mathcal T$,
$$  \prod_{q \sim e} x(q') = \prod_{q \sim e} (-x(q'')).$$
 It
can be proved that solutions to Thurston's equation over the real
numbers on $(M, \mathcal T)$ correspond to nowhere zero solutions
to HTE. The main observation in \cite{Luo1} is that critical
points of an entropy function of the form $\sum_{i=1} ^n x_i \ln(
|x_i|)$ are nowhere zero solutions to HTE.  The converse also
holds if $M$ is a closed 3-manifold.

Our pentagon relation (\ref{Pent2}) implies the following pentagon
relation for the entropy. Namely, given five positive numbers
$a_1, a_2, b_1, b_2, b_3$ so that $\sum_{i=1}^2a_i+\sum_{j=1}^3
b_i=1$ and $a_1a_2=b_1b_2b_3$, then \beq \label{pent3} \sum_{i, j}
(a_i+b_j) \ln(a_i+b_j) = \sum_{i=1}^2 ( a_i\ln(a_i)
+(1-a_i)\ln(1-a_i)) + \sum_{j=1}^3 b_j \ln(b_j). \eeq

Identity (\ref{pent3}) suggests there should exist a non-quantum
topological invariant for 3-manifold from simplical $SL(2,\mathbf
R )$ Chern-Simons theory. Furthermore, this invariant is the
semi-classical limit of $W_b(M; \mathcal T, \alpha)$ when $b$
degenerates.

\appendix

\section{Special functions}
\subsection{Faddeev's quantum dilogarithm}
Faddeev's quantum dilogarithm $\Phi_b(z)$ is defined by the integral
   \begin{equation}\label{eq:ncqdl}
\Phi_b(z) \equiv \exp \left( \int_{\R+\textup{i}0} \frac{e^{-2\textup{i}zw}dw}{4\sinh(wb)\sinh(w/b)w} \right),
    \end{equation}
in the strip $|\textup{Im} z|<|\textup{Im} c_b|$, where
$$
c_b = \textup{i} (b+b^{-1})/2.
$$

It is usefull to define
\begin{equation}\label{eq:cons}
\zeta_{inv}\equiv e^{\pi \textup{i}(1+2 c_b^2)/6}=e^{\pi \textup{i} c_b^2}
\zeta_o^2,\quad
\zeta_o \equiv e^{\pi \textup{i}(1-4c_b^2)/12}.
\end{equation}

\beqa
\text{symmetry} \label{fadgdual}  \quad&
 \Phi_b(z) = \Phi_{-b}(z) = \Phi_{1/b}(z) \,,  \\[0.5mm]
\text{functional equations}  \label{fadgfunrel} \quad&
 \Phi_b(z-\textup{i}b^{\pm1}/2) = (1+e^{2\pi b^{\pm1}z}) \Phi_b(z+\textup{i}b^{\pm1}/2) \,, \\[0.5mm]
\text{inversion property} \label{invers} \quad&
 \Phi_b(z) \Phi_b(-z) = \zeta^{-1}_{inv} e^{\pi \textup{i} z^2}  \,, \\[0.5mm]
 \label{fadgz}
 \text{zeros} \quad&
 z \in  \big\{ c_b + m\textup{i}b + n\textup{i}b^{-1};m,n\in\BZ^{\geq 0}\big\}
    \,, \\[0.5mm] \label{fadp}
 \text{poles} \quad&
 z \in  \big\{ -c_b - m\textup{i}b - n\textup{i}b^{-1};m,n\in\BZ^{\geq 0}\big\}
    \,, \\[0.5mm]
\label{fadgres}
 \text{unitarity} \quad&
  \overline{\Phi_b(z)} = 1/\Phi_b(\overline{z}).
%s_b(x)\sim\frac{i}{2\pi(x+i\fr{Q}{2})}\;\;
%{\rm near}\;\;x=-i\frac{Q}{2}
\eeqa

\subsection{The elliptic Gamma function}
The elliptic gamma function is defined by the formula
\begin{equation}\label{ellgamma}
\Gamma(z;p,q) \ = \ \prod_{i,j=0}^\infty \frac{1-z^{-1}p^{i+1}q^{j+1}}{1-zp^iq^j},
\end{equation}
and it satisfies the following properties
\beqa
\text{symmetry} \label{ellgdual}  \quad&
 \Gamma(z;p,q) = \Gamma(z;q,p) \,,  \\[0.5mm]
\text{functional equations}  \label{ellgfunrel} \quad&
 \Gamma(qz;p,q) = \theta(z;p) \Gamma(z;p,q), \\
 & \Gamma(pz;p,q) = \theta(z;q) \Gamma(z;p,q) \,, \\[0.5mm]
\text{reflection property} \label{ellgrefl} \quad&
 \Gamma(z;p,q) \; \Gamma(\frac{pq}{z};p,q) = 1  \,, \\[0.5mm]
 \label{ellgz}
 \text{zeros} \quad&
 z \in  \big\{ p^{i+1}q^{j+1};i,j\in\BZ^{\geq 0}\big\}
    \,, \\[0.5mm] \label{ellp}
 \text{poles} \quad&
 z \in  \big\{ p^{-i}q^{-j};i,j\in\BZ^{\geq 0}\big\}
    \,, \\[0.5mm]
\label{ellgres}
 \text{residue} \quad&
  \Res_{z=1} \Gamma(z;p,q) = -\frac{1}{(p;p)_\infty (q;q)_\infty}.
%s_b(x)\sim\frac{i}{2\pi(x+i\fr{Q}{2})}\;\;
%{\rm near}\;\;x=-i\frac{Q}{2}
\eeqa

Here $\theta(z;p)$ is a theta-function
$\theta(z;p) = (z;p)_\infty (pz^{-1};p)_\infty.$

\subsection{Some useful formulas}
Faddeev's quantum dilogarithm and the hyperbolic gamma functions are related via formula
\[
\gamma^{(2)}(- \textup{i}\sqrt{\omega_1\omega_2}(x+c_b);\omega_1,\omega_2)=\frac{e^{\textup{i}\pi x^2/2}}{\sqrt{\zeta_{inv}}\Phi_b(x)},
\]
where $b:=\sqrt{\frac{\omega_1}{\omega_2}}$.

Recall that the inversion relation (\ref{Inv}) for
$\gamma^{(2)}(x)$ is of the form
\[
\gamma^{(2)}(x;\omega_1,\omega_2)\gamma^{(2)}(\omega_1+\omega_2-x;\omega_1,\omega_2)=1
\]
and the complex conjugation property
\[
\overline{\gamma^{(2)}(z)}=\gamma^{(2)}(\bar z).
\]
If we define
\[
\mathcal{B}(u,v):=\frac{\gamma^{(2)}(u;\omega_1,\omega_2)\gamma^{(2)}(v;\omega_1,\omega_2)}{\gamma^{(2)}(u+v;\omega_1,\omega_2)}
\]
then it is easy to see that
\begin{equation}\label{eq:BPsi}
\mathcal{B}\left(\sqrt{-\omega_1\omega_2}x,\sqrt{-\omega_1\omega_2}y\right)=\Psi\left(\frac x2 +c_b,-\frac x2-c_b,y\right)
\end{equation}
where
\begin{equation}\label{eq:Psi}
\Psi(u,v,w):=\int_{\R}\frac{\Phi_b(u+x)}{\Phi_b(v+x)}e^{\textup{i}2\pi wx}dx,
\end{equation}
which is calculated as follows \cite{Faddeev:2000if}
\beq
\Psi(u,v,w) = \zeta_o\frac{\Phi_b(u-v-c_b) \Phi_b(w+c_b)}{\Phi_b(u-v+w-c_b)} e^{-2 \pi \textup{i} w(v+c_b)}.
\eeq

\end{document}